\documentclass[a4paper,11pt]{article}
\usepackage{makeidx}
\usepackage[mathscr]{eucal}
\usepackage[dvips]{color}
\usepackage{amssymb, amsfonts, amsmath, amsthm, graphicx, cases, comment} 
\usepackage{here}
\usepackage{tasks}
\DeclareMathOperator{\sech}{sech}
\DeclareMathOperator{\Ker}{Ker}

\newtheorem{proposition}{Proposition}[section]
\newtheorem{theorem}{Theorem}[section]

\newtheorem{lemma}[proposition]{Lemma}
\newtheorem{remark}{Remark}[section]

\newtheorem{notation}{Notation}[section]

\newtheorem*{thank}{Acknowledgments}

\newcommand{\R}{\mathbb{R}}

\newcommand{\T}{\mathbb{T}}

\newcommand{\Z}{\mathbb{Z}}

\numberwithin{equation}{section}
\newcommand{\p}{\partial}
\newcommand{\varphias}{Q}

\numberwithin{equation}{section}

\author{}
\date{}
\title{Pitchfork bifurcation at line solitons for nonlinear Schr\"{o}dinger equations on the product space
$\R \times \T$}
\author{
Takafumi Akahori, 
Yakine Bahri, \ 
Slim Ibrahim \ and Hiroaki Kikuchi}

\date{}
\begin{document}
\maketitle

\begin{abstract}
In this paper, we study the bifurcation problem from a line soliton 
 for a stationary nonlinear Schr\"{o}dinger equation on the product space $\R \times \T$. We extend earlier results to a larger class of the nonlinearity in the equation. The salient point of our analysis relies on a lower bound of solution to the ``auxiliary equation''  and then on the application of the Crandall-Rabinowitz argument.
\end{abstract}

\section{Introduction}
In this paper, we study the bifurcation problem from line solitons for the following stationary nonlinear Schr\"odinger equation posed on the product space $\mathbb{R}\times \mathbb{T}$:  
\begin{equation} \label{sp}
- \p_{x}^{2} u- \p_{y}^{2} u + \omega u 
- |u|^{p-1} u = 0,
\quad
(x,y)\in \mathbb{R}\times \mathbb{T},
\end{equation}
where $\mathbb{T}:= \mathbb{R}/2\pi \Z$ is the one dimensional torus, 
 $\omega > 0$, $p>1$ and $u$ is the unknown real-valued function on $\R \times \T$. 
 Introducing the function $\mathcal{F} \colon (0, \infty) \times H^{2}(\R \times \T) \to L^{2}(\R \times \T)$ defined by
\begin{equation}\label{eq-F}
\mathcal{F}(\omega, u) 
:= - \p_{x}^{2} u - \p_{y}^{2} u + \omega u 
- |u|^{p-1} u,
\end{equation}
equation \eqref{sp} can now be written as $\mathcal{F}(\omega, u) = 0$. We will use $\partial_{u}\mathcal{F}$ to denote the derivative of $\mathcal{F}=\mathcal{F}(\omega,u)$ with respect to the second variable $u$.

For each $\omega>0$, equation \eqref{sp} admits a positive solution in $H^{1}(\mathbb{R}\times \mathbb{T})$ which is independent of the variable $y \in \mathbb{T}$. Such a solution is called a line soliton to \eqref{sp}. In other words, a line soliton to \eqref{sp} is defined as a positive $H^{1}$-solution to the following ordinary differential equation:
\begin{equation}
\label{line-sp}
- \frac{d^{2}R}{dx^{2}} + 
\omega R - 
R^{p} = 0, 
\qquad x \in \R. 
\end{equation}
It is known that for any $\omega>0$ and $p>1$, a unique positive $H^{1}$-solution to \eqref{line-sp} exists; We use $R_{\omega}$ to denote the line soliton to \eqref{sp}. The line soliton $R_{\omega}$ is explicitly given as
\begin{equation}
\label{Rw:explicit}
R_{\omega}(x,y)=
R_{\omega}(x) := \omega^{\frac{1}{p-1}}
\Big(
\frac{p+1}{2}
\Big)^{\frac{1}{p-1}} 
\sech^{\frac{2}{p-1}}
\big(\frac{p-1}{2} \sqrt{\omega} x\big).
\end{equation}
Note that the function $\omega \in (0,\infty) \mapsto R_{\omega} \in H^{2}(\mathbb{R})$ is $C^{\infty}$. 
 Furthermore, let $L_{\omega,+,0}$ be the linearized operator around $R_{\omega}$ for the ODE \eqref{line-sp}, namely,  
\begin{equation}\label{21/11/6/10:2}
L_{\omega,+,0}:=-\frac{d^{2}}{dx^{2}} + \omega - p R_{\omega}^{p-1}.
\end{equation}  
Then, the following are known (see, e.g., Section 3 of \cite{Chang-Gustaf-Nakanishi-Tsai}):
\begin{enumerate}
\item     
\begin{equation}\label{21/10/24/11:29}
\sigma(L_{\omega,+,0})
=
\Big\{-\frac{\omega}{\omega_{p}}\Big\} \cup \{0\} \cup [0,\infty)
\quad 
\mbox{with}
\quad  
\omega_{p} := \frac{4}{(p-1)(p+3)}
.
\end{equation} 
\item 
All eigenvalues of $L_{\omega,+,0}$ are simple, and 
\begin{equation}\label{21/11/4/11:35}
L_{\omega,+,0} R_{\omega}^{\frac{p+1}{2}} = -\frac{\omega}{\omega_{p}} R_{\omega}^{\frac{p+1}{2}} ,
\qquad 
L_{\omega,+,0} \frac{dR_{\omega}}{dx} =0. 
\end{equation}
Note that $R_{\omega}^{\frac{p+1}{2}}$ is even and $\dfrac{dR_{\omega}}{dx}$ is odd. 
\end{enumerate}
By the Fourier series expansion with respect to the second variable $y$, we see that for any $\omega>0$ and $f\in H^{2}(\mathbb{R}\times \mathbb{T})$, 
\begin{equation}\label{21/10/12:8}
\partial_{u} \mathcal{F}(\omega, R_{\omega}) f 
= 
\sum_{n\in \mathbb{Z}} 
\big(L_{\omega,+,0}+ n^{2} \big) f_{n}(x) e^{iny}
\quad 
\mbox{with}~~ 
f_{n}(x):=\frac{1}{2\pi} \int_{0}^{2\pi} e^{-iny} f(x,y)  \,dy
.
\end{equation} 
Thus, we find from \eqref{21/10/24/11:29} and \eqref{21/10/12:8} that 
\begin{equation}\label{21/10/24/12:22}
\sigma(\partial_{u} \mathcal{F}(\omega, R_{\omega}))
=
\bigcup_{n\in \mathbb{Z}}\sigma(L_{\omega,+,0}+n^{2})
.
\end{equation}


The aim of this paper is to show that for any $p>1$, a ``pitchfork bifurcation'' occurs at the line soliton $R_{\omega_{p}}$ (see Theorem \ref{main-prop}). 
 Let us recall that the pair $(\omega_{p}, R_{\omega_{p}})$
 is called a bifurcation point for the equation $\mathcal{F}=0$ with respect to the curve $\omega  \mapsto (\omega, R_{\omega})$  if every neighborhood of $(\omega_{p}, R_{\omega_{p}})$ contains zeros of $\mathcal{F}$ not lying on the curve (see \cite{Crandall-Rabinowitz}). 
 The implicit function theorem shows that if $(\omega_{p}, R_{\omega_{p}})$ is the bifurcation point, then zero is an eigenvalue of $\partial_{u}\mathcal{F}(\omega_{p},R_{\omega_{p}})$. Since $\partial_{u}\mathcal{F}(\omega_{p},R_{\omega_{p}})$ is just the linearized operator around $R_{\omega_{p}}$ for \eqref{sp}, 
 we can say that a necessary condition for $(\omega_{p}, R_{\omega_{p}})$ to be the bifurcation point is that $R_{\omega_{p}}$ is degenerate. 
 In addition, by \eqref{21/10/24/11:29}, \eqref{21/11/4/11:35} and \eqref{21/10/24/12:22}, we see that if  $\omega=\omega_{p}$, then 
 the kernel of $\partial_{u} \mathcal{F}(\omega_{p},R_{\omega_{p}})$ contains the real-valued functions $\p_{x}R_{\omega_{p}}$, $R_{\omega_{p}}^{\frac{p+1}{2}}\cos{y}$ and $R_{\omega_{p}}^{\frac{p+1}{2}} \sin{y}$.  
 We introduce the spaces $L_{\rm{sym}}^{2}$ and $H_{\rm{sym}}^{2}$ as 
\begin{align*}
L^{2}_{\rm{sym}}&:=
\bigm\{u\in L^{2}(\R \times \T)
\colon u(x, y) = u(-x, y)=u(x, -y)
\bigm\},
\\[6pt]
H^{2}_{\rm{sym}}&:= H^{2}(\R \times \T) \cap L_{\rm{sym}}^{2}. 
\end{align*}
Note that $R_{\omega_{p}}^{\frac{p+1}{2}} \cos{y} \in H_{\rm{sym}}^{2}$, whereas $\partial_{x}R_{\omega_{p}}$ and $R_{\omega_{p}}^{\frac{p+1}{2}} \sin{y}$ fail to lie in $H_{\rm{sym}}^{2}$.
Thus, $(\omega_{p}, R_{\omega_{p}})$ is still a candidate for the the bifurcation point in the setting $H_{\rm{sym}}^{2}$. 
 Here, we remark that the information about branches bifurcating from $(\omega_{p},R_{\omega_{p}})$ is useful in the study of the stability of the degenerate line soliton $R_{\omega_{p}}$ (see \cite{Yamazaki2, KKP}). 


Now, we sate the main result of this paper. 
\begin{theorem}\label{main-prop}
For any $p >1$, the pair $(\omega_{p}, R_{\omega_{p}})$ is a pitchfork bifurcation point for $\mathcal{F}=0$ in $(0,\infty)\times H_{\rm{sym}}^{2}$  with respect to the curve $\omega \mapsto (\omega, R_{\omega})$; Precisely, there exist $a_{*}>0$, $\delta_{*}>0$ and a $C^{2}$-curve $a \in (-a_{*},a_{*}) \mapsto 
 (\omega(a), \varphias(a)) \in (\omega_{p}-\delta_{*},~\omega_{p}+\delta_{*}) \times H^{2}_{\rm{sym}}$ with the following properties:  

\begin{enumerate}

\item 
The set of zeros of $\mathcal{F}$ in 
$(\omega_{p}-\delta_{*}, \omega_{p}+\delta_{*})\times H_{\rm{sym}}^{2}$
 consists of two curves $\omega \in (0,\infty) \mapsto (\omega, R_{\omega})$ and $a \in (-a_{*},a_{*}) \mapsto (\omega(a), \varphias(a))$.

\item
The following hold for all $a\in (-a_{*},a_{*})$: 
\begin{equation}\label{21/11/10/14:59}
\omega(0)=\omega_{p},
\qquad 
\dfrac{d\omega}{da}(0)=0
,
\end{equation}
\begin{equation}\label{main-omega1}
\begin{split}
\frac{d^{2} \omega}{da^{2}}(0)
&=
\omega_{p}  \{ p(p-1)\}^{2}
\langle 
\mathbf{T}(\omega_{p}, 0)^{-1}
\big(
R_{\omega_{p}}^{p-2}
\{ \psi_{\omega_{p}}\cos{y} \}^{2}
\big)
,~ 
R_{\omega_{p}}^{p-2} 
\{ \psi_{\omega_{p}}\cos{y}\}^{2} 
\rangle 
\\[6pt]
&\quad 
+\frac{p(p-1)(p-2)\omega_{p}}{3}
\langle 
R_{\omega_{p}}^{p-3} 
,~ 
\{\psi_{\omega_{p}}\cos{y}\}^{4}
\rangle 
,
\end{split} 
\end{equation}
where $\mathbf{T}(\omega_{p}, 0)
:=
P_{\perp} \partial_{u}\mathcal{F}(\omega_{p},R_{\omega_{p}})|_{X_{2}}$; Note that $\mathbf{T}(\omega_{p}, 0)\colon X_{2}\to Y_{2}$ is a bijection.


\item
For any $a\in (-a_{*},a_{*})$, $\varphias(a)$ is positive and written as  
\begin{align}\label{expan-phi}
\varphias(a) = R_{\omega_{p}} + 
a  
C_{p} R_{\omega_{p}}^{\frac{p+1}{2}} \cos y 
+ 
O(a^{2}) \qquad 
\mbox{in $H^{2}(\R \times \T)$}, 
\end{align}
where $C_{p} > 0$ is a normalizing constant chosen such that   
\begin{equation}\label{L2-norm11}
\|C_{p} R_{\omega_{p}}^{\frac{p+1}{2}} \cos{y} 
\|_{L^{2}(\R \times \T)}^{2} 
=1
.
\end{equation}
In particular, $\varphias(0)=R_{\omega_p}$. 
 Furthermore, the mass of $\varphias(a)$ is written as 
\begin{equation}\label{eq1-14}
\begin{split}
&
\|\varphias(a)\|_{L^{2}(\R \times \T)}^{2} 
\\[6pt]
&= 
\pi \|R_{\omega_{p}}\|_{L^{2}(\R)}^{2} 
+\frac{a^{2}}{ \omega_{p}} 
\Big\{ 
\frac{d^{2}\omega}{da^{2}}(0)
\frac{5-p}{4(p-1)}  
\|R_{\omega_{p}}\|_{L^{2}(\R)}^{2} 
-1 
\Big\} 
+ o(a^{2})
.
\end{split} 
\end{equation}
\end{enumerate} 
\end{theorem}
\begin{remark}\label{21/8/8/15:47}
When $p\ge 2$, the same result as in Theorem \ref{main-prop} had been obtained in  \cite{Yamazaki2}. The salient point of our result is that we can treat all $p>1$; When $1<p<2$, the twice differentiability of the curve $(\omega(a),Q(a))$ is not obvious because the nonlinearity is not twice differentiable. 
\end{remark}

We give a proof of Theorem \ref{main-prop} in Section \ref{21/10/27/11:47}. In order to prove the theorem, we derive a lower bound of the solution to the ``auxiliary equation'' (see \eqref{au-eq}) first, 
 and then apply the Crandall-Rabinowitz argument \cite{Crandall-Rabinowitz}.  Such an approach enables us to treat the case $1<p<2$.

At the end of this section, we give a basic properties of line solitons to \eqref{sp}. It follows from the definition of the line solitons (see \eqref{Rw:explicit}) that   
\begin{align}
\label{21/11/3/17:9}
R_{\omega}(x)
&= \omega^{\frac{1}{p-1}} R_{1}(\sqrt{\omega}x),
\\[6pt] 
\label{21/11/5/17:4}
\p_{\omega}R_{\omega}
&=
\frac{1}{p-1}\omega^{-1}R_{\omega}
+
\frac{1}{2}\omega^{-1}x \frac{dR_{\omega}}{dx},
\\[6pt]
\label{decay-R}
\omega^{\frac{1}{p-1}}e^{-\sqrt{\omega}|x|} 
&\le 
R_{\omega}(x) 
\le
\{ 2(p+1)\omega \}^{\frac{1}{p-1}} e^{- \sqrt{\omega} |x|} 
.
\end{align}
Differentiating both sides of $\mathcal{F}(\omega, R_{\omega})=0$ with respect to $\omega$, we see that 
\begin{equation}\label{21/11/7/15:51}
\partial_{u}\mathcal{F}(\omega, R_{\omega}) \p_{\omega}R_{\omega}=-R_{\omega}
. 
\end{equation}

Furthermore, by \eqref{21/11/5/17:4} and the integration by parts,  
 we can verify that  
\begin{equation}\label{21/11/6/10:6}
\int_{\mathbb{R}} 
R_{\omega}^{q} \partial_{\omega}R_{\omega} 
= 
\frac{2q-p+3}{2(p-1)(q+1)}
 \omega^{-1} 
\int_{\mathbb{R}} 
R_{\omega}^{q+1} 
\qquad 
\mbox{for all $q\ge 1$}
.  
\end{equation}

We can also derive the following equation (see Lemma 2.2 of \cite{Yamazaki2}): 
\begin{equation}\label{21/11/6/14:34}
\int_{\R}
R_{\omega}^{p+r}
=
\frac{(p+1)(r+1)} {2r+p+1}
\omega
\int_{\mathbb{R}} 
R_{\omega}^{r+1} 
\qquad 
\mbox{for all $r>1$}
.
\end{equation}

The rest of this paper is 
organized as follows. 
 In Section \ref{21/11/14/11:30}, 
 we apply the Lyapunov-Schmidt method and introduce the auxiliary and bifurcation equations (see \eqref{au-eq} and \eqref{bifur-eq}). 
 In Section \ref{21/11/14/12:12}, we derive a lower bound and  decay estimates for solutions to the auxiliary equation. 
In Section \ref{sec-regularlity}, we show the three times differentiability of solutions to the auxiliary equation with respect to certain parameters.  
 In Section \ref{21/10/27/11:47}, we give a proof of Theorem \ref{main-prop}. 



\section{Lyapunov-Schmidt method}\label{21/11/14/11:30}

In order to construct a bifurcation branch of $\mathcal{F}=0$ from $(\omega_{p},R_{\omega_{p}})$ in $(0,\infty)\times H^{2}_{\rm{sym}}(\R \times \T)$, we employ the Lyapunov-Schmidt 
method.  Let us begin by introducing a few symbols: 
\begin{notation}
\begin{enumerate}
\item  
For $\omega>0$, define $\psi_{\omega}$ as 
\begin{equation}\label{eq1-10}
\psi_{\omega}:= \dfrac{R_{\omega}^{\frac{p+1}{2}}}{\| R_{\omega}^{\frac{p+1}{2}} \cos{y} \|_{L^{2}(\R\times \T)}}
=
\dfrac{R_{\omega}^{\frac{p+1}{2}}}{\sqrt{\pi}\| R_{\omega}^{\frac{p+1}{2}}\|_{L^{2}(\R)} 
}
.
\end{equation}
Observe from \eqref{decay-R} that 
\begin{equation} \label{eq1-11}
|\psi_{\omega_{p}}(x)| \lesssim 
e^{- \frac{p+1}{2} \sqrt{\omega} |x|}
\qquad \mbox{for all $x \in \mathbb{R}$}, 
\end{equation}
where the implicit constant depends only on $p$.
\item 
For $p>1$ and $\omega>0$, let $\lambda(\omega)$ denote the second eigenvalue of the operator $\partial_{u}\mathcal{F}(\omega, R_{\omega})$ restricted to $H_{\rm{sym}}^{2}$.  
 Note that $\partial_{u}\mathcal{F}(\omega, R_{\omega})|_{H^{2}_{\rm{sym}}}$ is a sefl-adjoint operator on $L^{2}(\R \times \T)$ with the values in $L^{2}_{\rm{sym}}$.
By \eqref{21/10/24/12:22}, \eqref{21/10/24/11:29}, \eqref{21/11/4/11:35} and \eqref{eq1-10}, we see that   
\begin{equation}
\label{21/11/4/11:58}
\lambda(\omega)=1-\frac{\omega}{\omega_{p}},
\quad 
\partial_{u}\mathcal{F}(\omega, R_{\omega})|_{H_{\rm{sym}}^{2}}
\psi_{\omega} \cos{y}
=
\lambda(\omega)  
\psi_{\omega} \cos{y} 
\qquad 
\mbox{for all $\omega >0$}
.
\end{equation}
Furthermore, we see that 
\begin{align}
\label{eq1-12}
&
\Ker \partial_{u}\mathcal{F}
(\omega_{p}, R_{\omega_{p}}) |_{H^{2}_{\rm{sym}}} 
= 
\operatorname{span}{\{ 
\psi_{\omega_{p}} \cos{y}
\}}, 
\\[6pt]
\label{21/11/4/13:17}
&
\frac{d\lambda}{d\omega}(\omega)
=
-\frac{1}{\omega_{p}}
\qquad
\mbox{for all $\omega >0$}
. 
\end{align}

\item 
We use $L_{\rm{real}}^{2}(\R \times \T)$ to denote the real Hilbert space of square integralbe functions on $\R\times \T$ equipped with the inner product
\[
\langle u, v \rangle 
:=
\int_{\mathbb{R}^{d} \times [0,2\pi]} 
u(x,y) v(x,y) \,dxdy .
\]

\item 
We define the spaces $X_{2}$ and $Y_{2}$ as  
\begin{align*}
X_{2} &:=
\{
u \in H^{2}_{\rm{sym}}
(\R \times \T) \colon 
\langle u, \psi_{\omega_{p}} \cos y \rangle 
= 0
\}, 
\\[6pt]
Y_{2} &:= 
\{
u \in L^{2}_{\rm{sym}}
(\R \times \T) \colon 
\langle u, \psi_{\omega_{p}} \cos y \rangle 
= 0
\}. 
\end{align*}
Note that $X_{2}\subset Y_{2}$. 
 Since the line solitons are independent of $y$, 
 it is easy to verify that  
\begin{equation}\label{21/7/21/16:37}
R_{\omega}, \p_{\omega}R_{\omega} \in X_{2} 
\qquad 
\mbox{for all $p>1$ and $\omega>0$}
. 
\end{equation}
\end{enumerate}
\end{notation}

We look for solutions to the equation $\mathcal{F}=0$ of the form $(\omega,u)=(\omega_{p}+\delta, R_{\omega_{p}}+a \psi_{\omega_{p}}\cos{y}+h)$ with $\delta>0$, $a\in \R$ and $h\in H_{\rm{sym}}^{2}$. 
 To this end, we consider the orthogonal projection $P_{\perp}$ from $L^{2}(\R \times \T)$ onto $Y_{2}$:   
\begin{equation}\label{21/8/5/17:39}
P_{\perp} u := u - \langle u, \psi_{\omega_{p}} \cos y \rangle 
\psi_{\omega_{p}} \cos y ,
\end{equation} 
Then, we define the function $\mathcal{F}_{\perp} \colon (0, \infty) \times (-1,1) \times H_{\rm{sym}}^{2} 
 \to Y_{2}$ as  
\begin{equation}\label{21/8/5/14:17}
\mathcal{F}_\perp(\omega, a, h)
:= P_\perp \mathcal{F}
(\omega, 
R_{\omega_{p}} + a \psi_{\omega_{p}} 
\cos y + h) 
. 
\end{equation}
Note that 
\begin{equation}\label{21/8/5/14:19}
\mathcal{F}_{\perp}
(\omega,~0,~R_{\omega} - R_{\omega_{p}}) 
=
P_\perp \mathcal{F}
(\omega, 
R_{\omega}) 
= 0
. 
\end{equation}
Note that \eqref{eq1-12} shows that 
\begin{equation}\label{21/7/24/16:40}
\Ker \partial_{u}\mathcal{F}
(\omega_{p}, R_{\omega_{p}}) |_{X_{2}} 
=\{0\} .
\end{equation} 
Furthermore, by the decay of $R_{\omega}$ (see \eqref{decay-R}), \eqref{21/7/24/16:40} and the Fredholm alternative theorem, we see that 
\begin{equation}\label{orthgo-1}
\operatorname{Ran}{\partial_{u}\mathcal{F}
(\omega_{p}, R_{\omega_{p}}) |_{X_{2}}}
= Y_{2}.  
\end{equation}
By \eqref{21/7/24/16:40} and \eqref{orthgo-1}, we see that 
$\partial_{u}\mathcal{F}(\omega_{p}, R_{\omega_{p}}) \colon X_{2} \to Y_{2}$ is bijective.

Using the implicit function theorem, we can find 
 solutions to $\mathcal{F}_{\perp}(\omega,a,h)=0$:
\begin{lemma}\label{main-lem-2}
Assume $p>1$. Then, there exist $a_{0}>0$, $\delta_{0} > 0$, $r_{0}>0$  
and a $C^1$-function  
$\eta \colon 
(\omega_{p} - \delta_{0},~\omega_{p}+ \delta_{0}) 
\times (-a_{0}, a_{0}) \to X_{2}$ such that the following hold for all $(\omega, a)\in (\omega_{p} - \delta_{0},~\omega_{p}+ \delta_{0}) \times (-a_{0}, a_{0})$:
\begin{enumerate}
\item Let $h \in X_{2}$. Then, $\mathcal{F}_{\perp}
(\omega, a, h) = 0$~if and only if~$h = \eta(\omega, a)$. 

\item  
$\|\eta(\omega,a)\|_{H^{2}(\R \times \T)} < r_{0}$. 

\item Define $\varphi(\omega,a)$ as 
\begin{equation}\label{eq2-10}
\varphi(\omega, a)
:= 
R_{\omega_{p}} + a \psi_{\omega_{p}} \cos y + \eta(\omega, a). 
\end{equation}
Then, the following hold: 
\begin{align}
\label{21/11/5/10:6}
\partial_{\omega} \eta(\omega,a)
&=
-
\big( P_{\perp} \partial_{u} \mathcal{F}(\omega, \varphi(\omega,a))|_{X_{2}} \big)^{-1}
P_{\perp} \varphi(\omega,a)
,
\\[6pt]
\label{21/11/5/11:11} 
\partial_{a} \eta(\omega,a)
&=
-\big( P_{\perp} \partial_{u} \mathcal{F}(\omega, \varphi(\omega,a))|_{X_{2}} \big)^{-1}
P_{\perp} \partial_{u} \mathcal{F}(\omega, \varphi(\omega,a))
\psi_{\omega_{p}}\cos{y}
. 
\end{align} 
\end{enumerate}
\end{lemma}
\begin{remark}\label{21/8/15/18:6}
\begin{enumerate} 

\item 
Lemma \ref{main-lem-2} only states that the function $\eta$ is $C^{1}$ since the assumption includes the case $1<p< 2$. When $p\ge 2$, we can prove that $\eta$ is $C^{2}$. 

\item 
Since $\eta$ is $C^{1}$ with respect to $(\omega, a)$ in $H^{2}(\R \times \T)$, replacing $\delta_{0}$ by $\dfrac{\delta_{0}}{2}$ and $a_{0}$ with $\dfrac{a_{0}}{2}$ if necessary, we may assume that 
\begin{equation}\label{21/10/29/9:59}
\sup_{\omega \in (\omega_{p}-\delta_{0}, \omega_{p}+\delta_{0})}
\sup_{a\in (-a_{0}, a_{0})}
\big\{ \|\partial_{\omega}\eta(\omega,a)\|_{H^{2}(\R\times \T)}
+ 
\|\partial_{a}\eta(\omega,a)\|_{H^{2}(\R\times \T)} \big\} 
\lesssim 1,
\end{equation}
where the implicit  constant depends only on $p$.
\end{enumerate}
\end{remark}

Let $a_{0}>0$, $\delta_{0} > 0$, $r_{0}>0$, $\eta$ and $\varphi$ be the same as in Lemma \ref{main-lem-2}. 
 By \eqref{21/7/21/16:37}, \eqref{21/8/5/14:19}, Lemma \ref{main-lem-2}, we see that 
\begin{equation}\label{phi-a=01}
\eta(\omega, 0) = R_{\omega} - R_{\omega_{p}}
\quad 
\mbox{for all $\omega \in (\omega_{p} - \delta_{0},~\omega_{p}+ \delta_{0})$}.
\end{equation}
In particular, we see that 
\begin{equation}\label{phi-a=0}
\eta(\omega_{p}, 0) =0, 
\qquad 
\varphi(\omega, 0) =R_{\omega},
\qquad 
\p_{\omega}\eta(\omega,0)
=
\p_{\omega}\varphi(\omega,0)
=
\p_{\omega}R_{\omega}
.
\end{equation}
We can also verify that 
\begin{equation}\label{21/8/1/15:1}
\p_{a} \eta(\omega_{p}, 0) = 0,
\qquad 
\p_{a} \phi(\omega_{p},0)=\psi_{\omega_{p}}\cos{y}
. 
\end{equation}  

Now, we introduce the function $\mathcal{F}_{\parallel}\colon 
(\omega_{p} - \delta_{0},~\omega_{p}+ \delta_{0}) \times (-a_{0}, a_{0}) 
\to \mathbb{R}$ as 
\begin{equation}\label{F-para}
\mathcal{F}_{\parallel} (\omega, a)
:= 
\langle  
\mathcal{F}(\omega, \varphi(\omega, a)),~
\psi_{\omega_{p}} \cos{y} 
\rangle
. 
\end{equation}
Then, we may write the equation 
 $\mathcal{F}(\omega, R_{\omega_{p}}+a \psi_{\omega_{p}}\cos{y}+h) = 0$ as follows:
\begin{numcases}{}
\mathcal{F}_{\perp}(\omega,a, h) = 0, \label{au-eq}
\\
\mathcal{F}_{\parallel}(\omega,a)= 0.  \label{bifur-eq}
\end{numcases}
The first equation \eqref{au-eq} is called the auxiliary equation and the second one \eqref{bifur-eq} the bifurcation equation. 
 Lemma \ref{main-lem-2} shows that  $(\omega,a,\eta(\omega,a))$ is a solution to the auxiliary equation \eqref{au-eq} for all $(\omega, a)\in (\omega_{p} - \delta_{0},~\omega_{p}+ \delta_{0}) \times (-a_{0}, a_{0})$. 
 Observe from \eqref{phi-a=0}, 
\eqref{F-para} and 
$\mathcal{F}(\omega, R_{\omega}) = 0$ that 
\begin{equation} \label{eq2-12}
\mathcal{F}_{\parallel}(\omega, 0) = 0 
.
\end{equation}


\section{Lower bound and decay of solution to the auxiliary equation}
\label{21/11/14/12:12}

Throughout this section, for a given $p>1$, we use $a_{0}$, $\delta_{0}$ and $\eta \colon (\omega_{p}-\delta_{0}, \omega_{p}+\delta_{0}) \times (-a_{0},a_{0}) \to X_{2}$ provided by Lemma \ref{main-lem-2}. 
 Furthermore,  let $\varphi(\omega,a)$ be the function defined by \eqref{eq2-10}, namely, $\varphi(\omega,a)=R_{\omega_{p}}+a \psi_{\omega_{p}}\cos{y}+ \eta(\omega,a)$.

Our aim in this section is to show the following propositions: 
\begin{proposition}\label{lem-up-low-b}
Assume $p > 1$. Then, there exists $\varepsilon_{1} > 0$ depending 
 only on $p$ with the following property: for any $0<\varepsilon <\varepsilon_{1}$, there exist $0< a_{\varepsilon} < a_{0}$ and $C_{1}(\varepsilon) 
> 1$ depending only on $p$ and $\varepsilon$ such that 
 for any $(\omega, a, x, y) 
\in (\omega_{p} - \delta_{0}, \omega_{p} 
+ \delta_{0}) \times 
(-a_{\varepsilon}, a_{\varepsilon}) \times \R \times \T$, 
\begin{equation}
\label{c-e-de}
\frac{1}{C_{1}(\varepsilon)} 
e^{-(\sqrt{\omega}+\varepsilon)|x|} \leq 
\varphi(\omega, a, x, y) 
\le C_{1}(\varepsilon) 
e^{-(\sqrt{\omega}-\varepsilon)|x|}
.
\end{equation}
In particular,  $\varphi(\omega, a)$ is positive. 
\end{proposition}

\begin{proposition} \label{c1-uni-decay}
Assume $p>1$. Let $\varepsilon_{1}>0$ be the constant given in 
 Proposition \ref{lem-up-low-b}.  
 Then, for any $0< \varepsilon < \varepsilon_{1}$, there exist $0< a_{\varepsilon} < a_{0}$, $C_{2}(\varepsilon)>0$ and $C_{3}(\varepsilon)>0$ depending only on $p$ and $\varepsilon$ such that 
 for any $(\omega, a, x, y) \in 
(\omega_{p} - \delta_{0}, \omega_{p} 
+ \delta_{0}) \times 
(-a_{\varepsilon}, a_{\varepsilon})
\times \R \times \T$, 
\begin{align}
\label{c-exp-decay3}
\big|\p_{\omega} \varphi
(\omega, a, x, y) 
\big|
\le C_{2}(\varepsilon) 
e^{-(\sqrt{\omega} - 2\varepsilon)|x|}, 
\\[6pt]
\label{c-exp-decay1}
\big|\p_{a} \varphi(\omega, a, x, y) \big| 
\le 
C_{3}(\varepsilon) e^{-(\sqrt{\omega}-\varepsilon)|x|} 
.
\end{align}
\end{proposition}
We use these propositions to 
show the differentiability 
of $\varphi(\omega, a)$ with respect to 
$\omega$ and $a$ (see Section 
\ref{sec-regularlity} below 
in details). 

We will introduce symbols used in the rest of this paper: 
\begin{notation}\label{21/7/29/13:59}
\begin{enumerate}
Let $p>1$ and $(\omega,a)\in (\omega_{p}-\delta_{0}, \omega_{p}+\delta_{0}) \times (-a_{0},a_{0})$.
\item 
Define the operators $\mathbf{L}_{-}(\omega, a)$ and $\mathbf{L}_{+}(\omega, a)$  on $L_{\rm{real}}^{2}(\R \times \T)$ with the domain $H^{2}(\R \times \T)$ as \begin{align}
\label{21/7/29/15:55}
\mathbf{L}_{-}(\omega, a) 
&:= 
- \p_{x}^{2} 
- \p_{y}^{2} + \omega 
- \varphi(\omega, a)^{p-1},
\\[6pt]
\label{def-L}
\mathbf{L}_{+}(\omega, a) 
&:= 
\partial_{u} \mathcal{F}(\omega,\varphi(\omega,a))
=
- \partial_{x}^{2} 
- \partial_{y}^{2} + \omega 
- p  \varphi(\omega, a)^{p-1}
. 
\end{align} 
Note that both $\mathbf{L}_{-}(\omega, a)$ and $\mathbf{L}_{+}(\omega, a)$ are self-adjoint operators on $L^{2}(\R\times \T)$ with domain $H^{2}(\R\times \T)$. Observe from \eqref{phi-a=0} that 
\begin{equation}\label{21/8/1/15:10}
\mathbf{L}_{+}(\omega, 0)
=
\partial_{u}\mathcal{F}(\omega, R_{\omega}) .
\end{equation}
Furthermore, by \eqref{21/8/1/15:10}, \eqref{phi-a=0} and \eqref{21/11/7/15:51}, we see that 
\begin{equation}\label{21/11/7/15:47}
\mathbf{L}_{+}(\omega, 0) \p_{\omega}\varphi(\omega,0)
=
\partial_{u}\mathcal{F}(\omega, R_{\omega}) \p_{\omega}R_{\omega} 
=
-R_{\omega}
.
\end{equation}

\item Define 
\begin{equation}\label{21/10/24/17:47}
E(\omega,a)
:= 
\mathcal{F}_{\parallel}(\omega,a) \psi_{\omega_{p}} \cos y
. 
\end{equation} 
Lemma \ref{main-lem-2} together with \eqref{eq2-10} shows  that 
\begin{equation}\label{21/7/29/16:17}
\mathbf{L}_{-}(\omega, a) \varphi(\omega, a) 
=
\mathcal{F}_{\perp}(\omega,a, \eta(\omega,a)) 
+
\mathcal{F}_{\parallel}(\omega,a) \psi_{\omega_{p}} \cos y
=
E(\omega,a)
. 
\end{equation} 
\end{enumerate}
\end{notation}

\subsection{Proof of Proposition \ref{lem-up-low-b}}
Following the argument of 
Berestycki and Nirenberg~\cite{BN}, 
we shall prove Proposition \ref{lem-up-low-b}. 
The maximal principle plays an 
important role in the proof. 
We first show 
a uniform decay of $\varphi(\omega, a)$. 
\begin{lemma} \label{uni-decay}
Assume $p > 1$. Then, for any $\omega \in 
(\omega_{p} - \delta_{0}, \omega_{p} + \delta_{0})$,  
the following holds: 
\begin{equation} \label{aim}
\lim_{|x| \to \infty} 
\sup_{a \in [-\frac{a_{0}}{2}, 
\frac{a_{0}}{2}]}\; 
\sup_{y \in \T} 
\varphi(\omega, a, x, y) = 0. 
\end{equation}
\end{lemma}
\begin{proof}
If the claim were false, then there existed $\omega_{*} \in (\omega_{p} - \delta_{0}, \omega_{p} + \delta_{0})$,  $C_{*}>0$ and a sequence $\{x_{n}\}$ with $\lim_{n \to \infty}|x_{n}| = \infty$ such that 
\[
\sup_{a \in [-\frac{a_{0}}{2}, 
\frac{a_{0}}{2}]}
\sup_{y \in \T}
|\varphi(\omega_{*}, a, x_{n}, y)| 
\ge C_{*}
\qquad \mbox{for all $n \in 
\mathbb{N}$}. 
\]
Furthermore, for each $n \in \mathbb{N}$, 
there exist $a_{n} \in 
[-\frac{a_{0}}{2}, 
\frac{a_{0}}{2}]$ 
and $y_{n} \in \T$
such that 
\begin{equation} \label{eq-1}
|\varphi(\omega_{*}, a_{n}, 
x_{n}, y_{n})| \geq \frac{C_{*}}{2}. 
\end{equation}
Since $[- \frac{a_{0}}{2}, \frac{a_{0}}{2}]$ is compact,  
we may assume that there exists 
$a_{\infty} \in [- 
\frac{a_{0}}{2}, 
\frac{a_{0}}{2}]$ such that 
$\lim_{n \to \infty}a_{n} = a_{\infty}$. 
Note that  $\partial_{a}\varphi(\omega_{*},a)$ is continuous with respect to $a$ in $H^{2}(\R \times \T)$. 
 Then, by the Sobolev inequality and 
the mean-value theorem, 
we see that  
\begin{equation}\label{21/7/28/11:54}
\begin{split} 
&
|\varphi(\omega_{*}, a_{n}, x_{n}, y_{n}) - 
\varphi(\omega_{*}, a_{\infty}, x_{n}, y_{n})| 
\lesssim 
\|\varphi(\omega_{*}, a_{n}) 
- \varphi(\omega_{*}, a_{\infty})
\|_{H^{2}(\R \times \T)} \\[6pt]
& 
\lesssim |a_{n} - a_{\infty}| 
\sup_{0\le \theta \le 1}
\big\|
\partial_{a} \varphi(\omega_{*}, a_{\infty}+\theta (a_{n}- a_{\infty}) )
\big\|_{H^{2}(\R \times \T)}
\lesssim |a_{n} - a_{\infty}|
,
\end{split}
\end{equation}
where the implicit constant is independent of $n$. 
We find from \eqref{21/7/28/11:54} that there exists a number $n_{*}\ge 1$ such that 
\begin{equation}\label{21/7/29/12:1}
|\varphi(\omega_{*}, a_{n}, x_{n}, y_{n}) - 
\varphi(\omega_{*}, a_{\infty}, x_{n}, y_{n})|
\leq \frac{C_{*}}{4} \qquad 
\mbox{for all $n \geq n_{*}$}. 
\end{equation}
Furthermore, it follows from \eqref{eq-1} and \eqref{21/7/29/12:1}
that 
\[
|\varphi(\omega_{*}, a_{\infty}, 
x_{n}, y_{n})| 
\geq \frac{C_{*}}{8} 
\qquad \mbox{for all $n \geq n_{*}$}, 
\]
which is absurd because 
$\varphi(\omega_{*}, a_{\infty}) \in H^{2}
(\R \times \T)$ implies that 
\[
\lim_{|x| \to \infty} \sup_{y \in \T} 
\varphi(\omega_{*}, a_{\infty}, x, y) = 0.
\] 
Thus, the claim of the lemma must be true. 
\end{proof}

Now, we give a proof of Proposition \ref{lem-up-low-b}: 
\begin{proof}[Proof of Proposition \ref{lem-up-low-b}]
Let $\omega \in (\omega_{p} - \delta_{0}, \omega_{p} + \delta_{0})$ and  
 $\varepsilon >0$ be a small constant to be specified later (see \eqref{21/10/26/11:30}), and define  
\begin{equation*} \label{eq-We}
W_{1}(x ,y)=W_{1}(x) := e^{-(\sqrt{\omega}-\varepsilon)|x|}, 
\qquad 
W_{2}(x, y)=W_{2}(x) := e^{-(\sqrt{\omega}+\varepsilon)|x|}.
\end{equation*}
Furthermore, let $a \in [-\frac{a_{0}}{2}, \frac{a_{0}}{2}]$ be a constant to be specified later. 
 It is easy to verify that
\begin{align}
\label{21/10/24/17:51}
& 
\mathbf{L}_{-}(\omega, a) W_{1} = \left( 2 \sqrt{\omega} 
\varepsilon - \varepsilon^2 
- |\varphi(\omega, a)|^{p-1} \right) W_{1}, 
\\[6pt]
\label{21/10/24/17:52}
& 
\mathbf{L}_{-}(\omega, a) W_{2} = \left(- 2 \sqrt{\omega} 
\varepsilon - \varepsilon^2 
- |\varphi(\omega, a)|^{p-1} \right) W_{2}.
\end{align}
Furthermore, we define 
\begin{equation*}
v_{1, a} := W_{1} - \varphi(\omega, a), 
\qquad  
v_{2, a} := \varphi(\omega, a)- W_{2}.
\end{equation*}
Then, by \eqref{21/10/24/17:51}, \eqref{21/10/24/17:52} and \eqref{21/7/29/16:17},  we see that 
\begin{align} 
\label{eq-L1} 
& 
\mathbf{L}_{-}(\omega, a) v_{1, a} = \left(2 \sqrt{\omega} \varepsilon 
- \varepsilon^2 
- |\varphi(\omega, a)|^{p-1}\right) W_{1} -E(\omega,a), 
\\[6pt] 
& 
\label{eq-L2}
\mathbf{L}_{-}(\omega, a) v_{2, a} = 
\left(2 \sqrt{\omega} \varepsilon + \varepsilon^2 
+ |\varphi(\omega, a)|^{p-1}\right) W_{2} +E(\omega,a)
.
\end{align}
Define 
\[
B_{r_{0}} 
:= \{ 
u \in X_{2}
\colon \|u\|_{H^{2}(\R \times \T)} < r_{0}
\}. 
\]
Then, by Lemma \ref{main-lem-2}, 
 there exists a constants $C_{*}>0$ depending only on $p$ 
 such that for any $x\in\R$, 
\begin{equation}\label{21/10/26/11:5}
\begin{split}
\sup_{y \in \T}|E(\omega,a, x, y)|
& 
\leq 
\|\mathcal{F}(\omega, 
\varphi(\omega, a))\|_{L^{2}(\R \times \T)}
\|\psi_{\omega_{p}} \cos y\|
_{L^{2}(\R \times \T)} 
|\psi_{\omega_{p}}(x)| 
\\ 
& 
\leq C_{*} e^{-\frac{p+1}{2} \sqrt{\omega}|x|}. 
\end{split}
\end{equation}

Now, we impose the smallness condition on $\varepsilon$ by 
\begin{equation}\label{21/10/26/11:30}
0<\varepsilon < \frac{\min\{p-1,~1\}}{4} \sqrt{\omega_{p}-\delta_{0}}.
\end{equation} 
Let $\rho (\varepsilon) > 0$ be a number such that 
\begin{equation}\label{21/10/27/10:14}
C_{*} e^{-\varepsilon \rho (\varepsilon)} \le \varepsilon^{2}.
\end{equation}  
 Then, by \eqref{21/10/26/11:5}, we see that 
 if $|x|\geq \rho (\varepsilon)$, then 
\begin{equation}\label{21/10/27/9:4}
\sup_{y \in \T}
|E(\omega,a,x, y)| 
\le
C_{*} e^{- 
\frac{p-1}{4}\sqrt{\omega_{p}-\delta_{0}} 
|x|
} W_{2} 
\leq C_{*} e^{- \varepsilon \rho (\varepsilon)}
W_{2}
\leq \varepsilon^{2} W_{2} 
\le
\varepsilon^{2} W_{1}
.
\end{equation} 
Furthermore, by Lemma \ref{uni-decay}, we may assume that  
 $\rho (\varepsilon) > 0$ is so large that 
\[
\sup_{a\in [-\frac{a_{0}}{2}, \frac{a_{0}}{2}]}
\sup_{y \in \T}
\big| \varphi(\omega, a, x, y) \big|^{p-1} 
\leq \varepsilon^2 
\qquad 
\mbox{for all $|x|\geq \rho (\varepsilon)$.}
\]  
Then, it follows from 
\eqref{eq-L1}, \eqref{eq-L2}, \eqref{21/10/27/9:4} and \eqref{21/10/26/11:30} 
 that   
\begin{equation}\label{positiv-Lv-1}
\mathbf{L}_{-}(\omega, a)v_{1, a} \ge  0, 
\qquad 
\mathbf{L}_{-}(\omega, a)v_{2, a} \ge 0 
\qquad 
\mbox{for all $|x| \ge \rho (\varepsilon)$ and $y \in \T$}
. 
\end{equation}
On the other hand, 
Since $\lim_{a \to 0} 
\|\varphi(\omega, a) - \varphi(\omega, 0)\|
_{L^{\infty}(\R \times \T)} = 0$ and $\varphi(\omega, 0) = 
R_{\omega}\ge \omega^{\frac{1}{p-1}}e^{-\sqrt{\omega_{p}+\delta_{0}}|x|}$ (see \eqref{phi-a=0} and \eqref{decay-R}),  
 we can take $0< a_{\varepsilon} \le \frac{a_{0}}{2}$ depending only on $p$ and $\varepsilon$ such that if $a \in (-a_{\varepsilon}, a_{\varepsilon})$, then 
\begin{equation}\label{21/10/27/10:6}
\frac{1}{2} R_\omega(x) \leq 
\varphi(\omega, a, x, y) \leq 2 
R_\omega(x) 
\qquad 
\mbox{for all $|x|\le 2\rho(\varepsilon)$ and $y\in \T$}.
\end{equation}
Furthermore, by \eqref{21/10/27/10:14}, \eqref{21/10/27/10:6} and \eqref{decay-R}, we see that 
for any $y\in \T$, 
\begin{equation}\label{21/10/24/10:17}
\begin{split} 
v_{1, a}(\rho (\varepsilon), y)
&\ge 
\varepsilon^{-2} C_{*}
e^{-\sqrt{\omega} \rho (\varepsilon)} 
-2\{ 2(p+1)\sqrt{\omega_{p}+\delta_{0}} \}^{\frac{1}{p-1}} 
e^{-\sqrt{\omega}\rho (\varepsilon)} 
.
\end{split}
\end{equation}
Similarly, we see that for any $y\in \T$,
\begin{equation}\label{21/10/27/10:18}
v_{2, a}(\rho (\varepsilon),y) 
\ge 
\frac{1}{2} (\omega_{p}-\delta_{0})^{\frac{1}{p-1}}e^{-\sqrt{\omega}\rho(\varepsilon)}
- 
C_{*}^{-1}\varepsilon^{2}
e^{-\sqrt{\omega} \rho (\varepsilon)}
.
\end{equation}
Thus, we find from \eqref{21/10/24/10:17} and \eqref{21/10/27/10:18} that 
 if $\varepsilon$ is sufficiently small dependently only on $p$, then 
\begin{equation}
\label{positiv-v-1}
v_{1, a} (\rho (\varepsilon), y) \geq 0, 
\qquad 
v_{2, a} (\rho (\varepsilon), y) \geq 0 
\qquad 
\mbox{for all $y\in \T$}
.
\end{equation}

Define  
\[
S_\varepsilon := 
\{ (x, y) \in \R \times \T \colon
|x| > \rho (\varepsilon) \}.
\]
We claim that
\begin{equation} \label{posi-max}
v_{1, a}, v_{2, a} \geq 0 
\qquad \mbox{for all $(x, y) 
\in S_\varepsilon$}.
\end{equation}
Note that \eqref{posi-max} together with \eqref{21/10/27/10:6} proves Proposition \ref{lem-up-low-b}. 
 We prove \eqref{posi-max} by contradiction. Suppose that 
\eqref{posi-max} fails. 
Then, we see that there exists 
$(x_{\min}, y_{\min}) \in \R \times [-\pi, \pi]$ 
such that 
$v_{1, a}(x_{\min}, y_{\min}) = 
\min_{(x, y) \in S_{\varepsilon}} 
v_{1, \alpha}(x, y) < 0$. 
By 
\eqref{positiv-Lv-1}, \eqref{positiv-v-1}
and  
$\omega -|\varphi(\omega, a)|^{p-1} \ge 0 $ on $S_{\varepsilon}$, 
we can apply the maximum principle 
and find that $y_{\min} \notin (-\pi, \pi)$, 
that is, $y_{\min} = \pi$ or $-\pi$. 
However, by the Hopf lemma 
and $\partial_y v_{1, a}(\pi)
= \partial_y v_{1, a}(-\pi)=0$, 
we can also find that 
$y_{\min} \neq \pm \pi$, which is 
absurd. 
Therefore, 
$v_{1, a} \geq 0$ 
for all $(x, y) \in S_\varepsilon$. 
Similarly, we can prove that  
 $v_{2, a} \geq 0$ on $S_\varepsilon$.  
Thus, we have completed the proof.  
\end{proof}
\subsection{Exponential decay of 
the derivatives of 
$\varphi(\omega, a)$}
In this subsection, we give a proof of Proposition \ref{c1-uni-decay}. 
To this end, for $(\omega,a)\in (\omega_{p}-\delta_{0}, \omega_{p}+\delta_{0}) \times (-a_{0},a_{0})$ and a given function $G$ on $\mathbb{R}\times \mathbb{T}$, we consider the following equation:
\begin{equation} \label{lsp-g}
\mathbf{L}_{+}(\omega, a) v 
= G
\qquad \mbox{in $\R \times \T$}
.
\end{equation}
By a standard Fourier analysis (see also Theorem 6.23 of \cite{Lieb-Loss}), we may write \eqref{lsp-g} as 
\begin{equation}\label{21/11/18/17:25}
v(x,y)
=C  
\int_{\mathbb{R}_{z}} 
\sum_{m\in \mathbb{Z}}
e^{imy}
\int_{\mathbb{R}}
\frac{e^{i\xi (x-z)}}{\xi^2 
+ m^{2} + \omega}
\int_{\T_{w}} e^{-im w}
\big\{
p|\varphi(\omega, a,z,w)|^{p-1}v(z,w) 
+ G(z,w)
\big\}dzd\xi dw
,
\end{equation}
where $C$ is some constant.
A key in proving 
Proposition \ref{c1-uni-decay} is the following: 
\begin{proposition} \label{thm-exd-1}
Assume $p>1$. Let $\varepsilon_{1}>0$ be the constant given in Proposition \ref{lem-up-low-b}, $0<\varepsilon< \varepsilon_{1}$, and let $a_{\varepsilon}$ denote the same constant as in Proposition \ref{lem-up-low-b}. Furthermore, let 
 $(\omega,a)\in (\omega_{p}-\delta_{0}, \omega_{p}+\delta_{0}) \times (-a_{\varepsilon},a_{\varepsilon})$, $v$ be a solution to \eqref{lsp-g}, $A>0$, $B>0$ and $0< \alpha <\sqrt{\omega}$.
 Assume that the function $G$ on the right-hand side of \eqref{lsp-g} obeys 
that 
\begin{equation} \label{decay-g}
|G(x, y)| \le A e^{- \alpha |x|}
\quad 
\mbox{for all $(x,y)\in \mathbb{R}\times \mathbb{T}$}
.
\end{equation} 
Furthermore, assume that 
\begin{equation}\label{21/10/30/15:33}
\|v \|_{L^{\infty}(\R\times \T)} \le B
.
\end{equation}
 Then, the following holds:  
\begin{equation}\label{decay-dphi}
|v(x, y)| \leq C_{\varepsilon} e^{- (\alpha - \varepsilon)|x|} 
\qquad \mbox{for all $(x, y) 
\in \R \times \T$},
\end{equation}
where the implicit constant depends only on $p$, $A$, $B$, $\alpha$ and $\varepsilon$. 
\end{proposition}

\begin{proof}[Proof of Proposition \ref{thm-exd-1}] 
By \eqref{21/11/18/17:25}, we see that 
\begin{equation}\label{21/10/28/16:36}
\begin{split}
&|v(x,y)| 
\\[6pt]
&\sim 
\Big|
\int_{\mathbb{R}_{z}} 
\sum_{m\in \mathbb{Z}}
e^{imy}
\frac{e^{-\sqrt{m^{2}+\omega}\,|x-z|}}{
\sqrt{m^{2} + \omega}}
\int_{\T_{w}} e^{-im w}
\big\{
p|\varphi(\omega, a,z,w)|^{p-1}v(z,w) 
+ G(z,w)
\big\}
\Big| 
. 
\end{split} 
\end{equation} 
Put  
\begin{equation}\label{21/10/28/16:37}
\nu_{m} := \sqrt{m^{2}+\omega}.
\end{equation}
Then, by \eqref{21/10/28/16:36}, \eqref{lem-up-low-b} and the assumptions \eqref{decay-g} and \eqref{21/10/30/15:33}, we see that 
\begin{equation}\label{eq-expo-0} 
\begin{split}
|v(x, y) |
\lesssim
B \int_{\mathbb{R}}
\sum_{m \in \mathbb{Z}} 
\frac{e^{- \nu_{m}|x - z|} }{\nu_{m}} 
e^{- (p-1)(\sqrt{\omega} - 
\varepsilon)|z|}
\, dz
+ 
A \int_{\mathbb{R}} 
\sum_{m \in \mathbb{Z}} 
\frac{e^{- \nu_{m}|x - z|} }{\nu_{m}}  
e^{- \alpha |z|} \, dz 
,
\end{split} 
\end{equation}
where the implicit constant depends only on $p$.

Consider the second term on the right-hand side of \eqref{eq-expo-0}. 
 Let $x\in \mathbb{R}$. 
 Then, for any $m\in \mathbb{Z}\setminus \{0\}$, a direct computation together with $0<\alpha<\sqrt{\omega} \le \nu_{m}$ and $\omega<\omega_{p}$ shows that 
\begin{equation}\label{21/10/30/11:37}
\int_{\R} 
e^{- \nu_{m}|x - z|} 
e^{- \alpha |z|} dz 
\le
\int_{\R} 
e^{- (\nu_{m}-\alpha) |x - z|}
e^{- \alpha (|x - z|+ |x|)} dz  
\le  
\frac{2e^{-\alpha|x|}}{\nu_{m}-\alpha} 
\lesssim 
\frac{e^{-\alpha|x|}}{m}
.
\end{equation}
where the implicit constant depends only on $p$. 
 Thus, by \eqref{21/10/30/11:37} and $\omega >\omega_{p}-\delta_{0}$, 
 we see that 
\begin{equation}\label{21/10/30/16:42} 
\int_{\R} 
\sum_{m \in \mathbb{Z}} 
\frac{e^{- \nu_{m}|x - z|} }{\nu_{m}} 
e^{- \alpha |z|} dz  
\lesssim  
\sum_{m \in \mathbb{Z} \setminus \{0\}} 
\frac{1}{m^{2}} e^{- \alpha|x|} 
+ 
\frac{1}{\sqrt{\omega}(\sqrt{\omega}-\alpha)} 
e^{- \alpha |x|}  
\lesssim 
e^{- \alpha  |x|}
,
\end{equation}
where the implicit constants depend only on $p$ and $\alpha$.

Consider the first term on the right-hand side of \eqref{eq-expo-0}. 
 Then, a computation similar to \eqref{21/10/30/11:37}  together with $\omega\ge \omega_{p}-\delta_{0}$ shows that 
\begin{equation}\label{21/10/30/17:10}
\begin{split}
&
\int_{\mathbb{R}}
\sum_{m \in \mathbb{Z}} 
\frac{e^{- \nu_{m}|x - z|} }{\nu_{m}} 
e^{- (p-1)(\sqrt{\omega} - 
\varepsilon)|z|}
dz
\\[6pt] 
& \lesssim 
\sum_{m \in \mathbb{Z}\setminus \{0\}} 
\frac{1}{m^{2}} 
e^{-\min\{1,~(p-1)\} (\sqrt{\omega} -\varepsilon) |x|} 
+ 
e^{-\min\{1,~(p-1)\} (\sqrt{\omega} -\varepsilon) |x|} 
\\[6pt]
&\lesssim 
e^{-\min\{1,~(p-1)\} (\sqrt{\omega} -\varepsilon) |x|}
,
\end{split}
\end{equation}
where the implicit constants depend only on $p$ and $\varepsilon$. 
Putting  \eqref{eq-expo-0}, \eqref{21/10/30/16:42} and \eqref{21/10/30/17:10} together, we find that 
\begin{equation}\label{21/10/30/17:31}
|v(x, y)| 
\lesssim 
B e^{-\min\{1,~(p-1)\} (\sqrt{\omega} -\varepsilon) |x|}
+
A e^{- (\alpha - \varepsilon)|x|} 
,
\end{equation}
where the implicit constant depends only on $p$, $\alpha$ and $\varepsilon$. 

When $\min\{1,~(p-1)\} (\sqrt{\omega} -\varepsilon) \le \alpha -\varepsilon$, 
 \eqref{21/10/30/17:31} implies the desired estimate \eqref{decay-dphi}.
 On the other hand, when $\min\{1,~(p-1)\} (\sqrt{\omega} -\varepsilon) > \alpha -\varepsilon$, using \eqref{21/10/30/17:31} in the computation \eqref{21/10/28/16:36}--\eqref{eq-expo-0} instead of the assumption \eqref{21/10/30/15:33}, we can verify that 
\begin{equation}\label{21/10/30/17:51}
|v(x, y)| 
\lesssim 
\max\bigm\{
e^{-\min\{1,~2(p-1)\} (\sqrt{\omega} -\varepsilon) |x|}
,~
e^{- (\alpha - \varepsilon)|x|} 
\bigm\},
\end{equation}
where the implicit constant depends only on $p$, $A$, $B$, $\alpha$ and $\varepsilon$. 
 Updating the bound of $|v|$ in \eqref{21/10/28/16:36}--\eqref{eq-expo-0} one after another, we see that  for any integer $k \ge 1$,
\begin{equation}\label{21/10/30/17:51}
|v(x, y)| 
\lesssim 
\max\bigm\{
e^{-\min\{1,~k(p-1)\} (\sqrt{\omega} -\varepsilon) |x|}
,~
e^{- (\alpha - \varepsilon)|x|} 
\bigm\},
\end{equation}
where the implicit constant depends only on $p$, $A$, $B$, $\alpha$, $\varepsilon$ and $k$. This implies \eqref{decay-dphi}.
\end{proof}


Now, we are in a position 
to prove Proposition \ref{c1-uni-decay}. 
\begin{proof}[Proof of Proposition \ref{c1-uni-decay}]
Note that 
$\p_{\omega} \varphi(\omega, a)$ and $\p_{a} \varphi(\omega, a)$
satisfy  
\begin{align*}
\mathbf{L}_{+}(\omega, a) 
\p_{\omega} \varphi(\omega, a)
= - \varphi(\omega,a) + \widehat{E}(\omega,a),
\qquad 
\mathbf{L}_{+}(\omega, a) 
\p_{a} \varphi(\omega, a)
= \widetilde{E}(\omega,a)
, 
\end{align*}
where 
\begin{align*}
\widehat{E}(\omega, a) 
&:= 
\langle 
\mathbf{L}_{+}(\omega, a) 
\p_{\omega} \varphi 
(\omega, a)+\varphi(\omega,a),~\psi_{\omega_{p}} 
\cos y
\rangle \psi_{\omega_{p}} \cos{y},
\\[6pt] 
\widetilde{E}(\omega,a) 
&:= 
\langle 
\mathbf{L}_{+}(\omega, a) 
\p_{a} \varphi(\omega, a), 
\psi_{\omega_{p}} \cos y
\rangle \psi_{\omega_{p}} \cos{y}
.
\end{align*}
Then, by Proposition \ref{lem-up-low-b}, \eqref{eq1-11}, 
 Lemma \ref{main-lem-2} and \eqref{21/10/29/9:59}, we see that 
 $\varphi(\omega, a)$, $\widehat{E}(\omega, a)$ and $\widetilde{E}(\omega,a)$ have exponential decay with respect to $x$: 
\begin{equation}\label{21/10/29/8:56}
|\varphi(\omega,a) |
\lesssim 
e^{-(\sqrt{\omega}-\varepsilon)|x|}, 
\quad 
|\widehat{E}(\omega, a) |
+
|\widetilde{E}(\omega,a)|
\lesssim 
e^{-\frac{p+1}{2}\sqrt{\omega}|x|}
\lesssim
 e^{-(\sqrt{\omega}-\varepsilon)|x|}
,
\end{equation} 
where the implicit constants depend only on $p$ and $\varepsilon$. 
Applying Proposition \ref{thm-exd-1} 
as $G = - \varphi(\omega,a) + \widehat{E}(\omega,a)$ and $G = \widetilde{E}(\omega,a)$, 
 we find that Proposition \ref{c1-uni-decay} is true. 
\end{proof}


\section{Computation of derivatives}
\label{sec-regularlity}
 The aim of this section is to compute the second and third derivatives of $\varphi(\omega,a)$ and $\mathcal{F}_{\parallel}(\omega, a)$ with respect to $a$ and $\omega$, 
 which are used in the application of the Crandall-Rabinowitz argument~\cite{Crandall-Rabinowitz} (see Section \ref{21/10/27/11:47}). 
 Note that when $1<p<2$, even twice differentiability is not obvious, as the nonlinearity of \eqref{sp} is not twice differentiable.

Throughout this section, let $a_{0}>0$ and 
$\delta_{0}>0$ be the constants given in Lemma \ref{main-lem-2}, and let 
  $\varepsilon_{1}>0$ be the constant given in Proposition \ref{lem-up-low-b}.
 Furthermore, for $0<\varepsilon < \varepsilon_{1}$, we use $a_{\varepsilon}$ to denote the same constant as in Proposition \ref{lem-up-low-b}. 
 We will assume that $\delta_{0}$ and $\varepsilon_{1}$ are  sufficiently small dependently only on $p$ without any notice. 

Recall from Proposition \ref{c1-uni-decay} that 
 if $p> 1$ and $0< \varepsilon < \varepsilon_{1}$, then the following holds for all $(\omega, a) \in (\omega_{p}-\delta_{0},~\omega_{p}+\delta_{0}) \times (-a_{\varepsilon}, a_{\varepsilon})$ and $(x, y)\in \R \times \T$: 
\begin{equation}\label{9/7-9:33}
| \partial_{a} \varphi (\omega, a, x, y)| 
+ |\partial_{\omega} \varphi (\omega, a, x, y)|
\lesssim 
e^{-\sqrt{\omega}|x|+ 2\varepsilon|x|}, 
\end{equation} 
where the implicit constant depends only on $p$ and $\varepsilon$. 


We introduce symbols which are used in this and the next sections:

\noindent 
{\bf Notation.}
\begin{enumerate}
\item 
For $p>1$, $k\ge 1$ and $(\omega, a) \in (\omega_{p}-\delta_{0}, \omega_{p}+\delta_{0}) 
\times (-a_{0}, a_{0})$, we define $V_{k}(\omega, a)$ as  
\begin{equation}\label{def-V}
V_{k} (\omega, a):= 
\frac{d^{k}x^{p}}{dx^{k}}\Big|_{x=\varphi(\omega, a)}
= p(p-1) \cdots 
(p-k+1) \varphi(\omega, a)^{p-k}.
\end{equation} 
Note that 
\begin{equation}\label{21/11/3/13:38}
\mathbf{L}_{+}(\omega, a) 
= -\partial_{x}^{2} - \partial_{y}^{2} + \omega -V_{1}(\omega, a).
\end{equation} 
 Since $\varphi(\omega, a)$ is positive 
 (see Proposition \ref{lem-up-low-b}) and of class $C^{1}$ on 
$(\omega_{p}-\delta_{0}, \omega_{p}+\delta_{0}) 
\times (-a_{0}, a_{0})$ 
in the $H^{2}(\R \times \T)$-topology 
(see Lemma \ref{main-lem-2} 
and \eqref{eq2-10}), 
 the following holds: 
\begin{equation}\label{v-k-deriv}
\partial V_{k}(\omega, a)
=
V_{k+1}(\omega, a) 
\partial \varphi(\omega, a)
\qquad 
\mbox{everywhere in 
$\mathbb{R}\times \mathbb{T}$}, 
\end{equation}
where 
$\partial$ denotes 
either $\partial_{a}$ or $\partial_{\omega}$. 
 Furthermore, Proposition \ref{lem-up-low-b} shows that 
 if $0< \varepsilon < \varepsilon_{1}$ and $(\omega, a) \in (\omega_{p}-\delta_{0},~\omega_{p}+\delta_{0}) \times (-a_{\varepsilon}, a_{\varepsilon})$, 
 then  
\begin{equation}\label{9/7-9:32}
V_{k}(\omega, a, x, y)
\lesssim 
\varphi(\omega, a ,x, y)^{p-k} 
\lesssim
e^{(k-p) \sqrt{\omega} |x| + \varepsilon 
(k +p) |x|}, 
\end{equation}
where the implicit constants 
depend only on $p$, $\varepsilon$ and $k$. 

\item 
For $p>1$ and $(\omega, a) \in (\omega_{p}-a_{0}, \omega_{p}+a_{0}) \times (-a_{0}, a_{0})$, we define $\mathbf{T}(\omega, a)$ as  
\begin{equation*}\label{def-T}
\mathbf{T}(\omega, a)
:=
P_{\perp} \mathbf{L}_{+}
(\omega, a)|_{X_{2}}
. 
\end{equation*}
By \eqref{21/8/1/15:10}, we see that 
\begin{equation}\label{21/8/1/15:15}
\mathbf{T}(\omega_{p}, 0)
=
P_{\perp} \partial_{u}\mathcal{F}(\omega_{p},R_{\omega_{p}})|_{X_{2}}, 
\end{equation}
We have to pay attention 
to the difference between 
$\mathbf{T}(\omega, a)$ and 
$P_{\perp} \mathbf{L}_{+}(\omega, a)$; 
 In particular, 
$\mathbf{T}(\omega, a)$ has the inverse, 
but $P_{\perp} \mathbf{L}_{+}(\omega, a)$ does not. 
\end{enumerate} 


\subsection{Basic results}\label{basic}

It is easy to verify that for any $(\omega_{1},a_{1}), (\omega_{2},a_{2}) \in (\omega_{p}-\delta_{0}, \omega_{p}+\delta_{0}) \times (-a_{0}, a_{0})$, 
\begin{align}
\label{29-1}
\mathbf{T}(\omega_{1}, a_{1}) - \mathbf{T}(\omega_{2}, a_{2})
&=
- 
P_{\perp} \{ V_{1}(\omega_{1},a_{1}) - V_{1}(\omega_{2}, a_{2}) \}
, 
\\[6pt]
\label{29-2}
\mathbf{T}(\omega_{1}, a_{1})^{-1} - \mathbf{T}(\omega_{2}, a_{2})^{-1}
&=
\mathbf{T}(\omega_{1}, a_{1})^{-1}\{ \mathbf{T}(\omega_{2}, a_{2}) - \mathbf{T}(\omega_{1}, a_{1}) \}\mathbf{T}(\omega_{2}, a_{2})^{-1}
.
\end{align}
Furthermore, it is known that for any $p>1$ and $(\omega, a) \in (\omega_{p}-\delta_{0}, \omega_{p}+\delta_{0}) \times (-a_{0}, a_{0})$, 
\begin{equation}\label{28-3}
\| \mathbf{T}(\omega, a)^{-1} \|_{L^{2}
(\R \times \T) \to H^{2}(\R \times \T)}
\lesssim 1
, 
\end{equation}
where the implicit constant depends only on $p$, which together with the linearity 
implies the continuity, namely if $\lim_{n\to \infty}f_{n} =f$ in $L^{2}(\R \times \T)$, then 
\begin{equation}\label{28-4}
\lim_{n \to \infty} 
\mathbf{T}(\omega, a)^{-1}f_{n} 
=
\mathbf{T}(\omega, a)^{-1} f
\quad 
\mbox{in $H^{2}(\R \times \T)$}
. 
\end{equation}

By the boundedness of $P_{\perp}$ in $L^{2}( \R \times \T)$, the continuity of $\varphi(\omega, a)$ with respect to $(\omega, a)$
(see Lemma \ref{main-lem-2} and \eqref{eq2-10}), and \eqref{29-1} through \eqref{28-3}, we can obtain the following lemma: 
\begin{lemma}\label{c-L}
Assume $p>1$. Then, the operators $\mathbf{T}(\omega, a) \colon X_{2} \to Y_{2}$ and $\mathbf{T}(\omega, a)^{-1} \colon Y_{2} \to X_{2}$ are continuous with respect to $(\omega, a)$ on $(\omega_{p}-\delta_{0},~\omega_{p}+\delta_{0}) \times (-a_{0}, a_{0})$, namely,  
\begin{align}
\label{c-L-1}
\lim_{(\gamma_{1}, \gamma_{2}) \to (0,0)}
\| \mathbf{T}(\omega+\gamma_{1}, a+\gamma_{2}) - \mathbf{T}(\omega, a) \|_{H^{2}(\R \times \T) \to L^{2}
(\R \times \T)} =0,
\\[6pt] 
\label{c-L-2}
\lim_{(\gamma_{1}, \gamma_{2}) \to (0,0)}
\|\mathbf{T}(\omega+\gamma_{1}, a+\gamma_{2})^{-1} - \mathbf{T}(\omega, a)^{-1} \|_{L^{2}_{x, y}(\R \times \T) 
\to H^{2}(\R \times \T)} =0.
\end{align} 
\end{lemma}

It is easy to verify that the following lemma holds: 
\begin{lemma}\label{p-d}
If $f \colon (\omega_{p}-\delta_{0}, \omega_{p}+\delta_{0}) \times (-a_{0}, a_{0}) \to L^{2}(\R \times \T)$ has the partial derivatives in $L^{2}(\R \times \T)$ ($\partial_{a}f (\omega, a), \partial_{\omega}f(\omega, a) \in L^{2}(\R \times \T)$), then, 
\begin{align*}
\partial_{a} \langle f(\omega, a), \psi_{\omega_{p}}\cos{y} \rangle 
&=
\langle \partial_{a}f(\omega, a), \psi_{\omega_{p}}\cos{y} \rangle 
, 
\\[6pt]
\partial_{\omega} \langle f(\omega, a), 
\psi_{\omega_{p}}\cos{y} \rangle 
&=
\langle \partial_{\omega}f(\omega, a), 
\psi_{\omega_{p}}\cos{y} \rangle. 
\end{align*}
In particular, the following holds in 
$L^{2}(\R \times \T)$:
\begin{align*}
\partial_{a}\{P_{\perp} f(\omega, a)\}
&=
P_{\perp} \partial_{a}f(\omega, a)
, 
\\[6pt]
\partial_{\omega}\{P_{\perp} f(\omega, a)\}
&=
P_{\perp} \partial_{\omega}f(\omega, a).
\end{align*}
\end{lemma}

Since $-\partial_{x}^{2}-\partial_{y}^{2} \colon H^{2}(\R \times \T) \to L^{2}(\R \times \T)$ is bounded (hence continuous), it is easy to verify that the following lemma holds: 
\begin{lemma}\label{d-lap}
Let $f \colon (\omega_{p}-\delta_{0}, \omega_{p}+\delta_{0}) \times (-a_{0}, a_{0}) \to H^{2}(\R \times \T)$ has the partial derivatives in $H^{2}(\R \times \T)$ ($\partial_{a}f (\omega, a), \partial_{\omega}f(\omega, a) \in H^{2}(\R \times \T)$), then, 
the following hold in $L^{2}(\R \times \T)$ for all $(\omega, a) \in (\omega_{p}-\delta_{0}, \omega_{p}+\delta_{0}) \times (-a_{0},a_{0})$:
\begin{align*}
\partial_{a}
\{(-\partial_{x}^{2}-\partial_{y}^{2}+\omega) f(\omega, a) \}
&=
(-\partial_{x}^{2}-\partial_{y}^{2}+\omega) \partial_{a} f(\omega, a), 
\\[6pt]
\partial_{\omega}
\{(-\partial_{x}^{2}-\partial_{y}^{2}+\omega) f(\omega, a) \}
&=
(-\partial_{x}^{2}-\partial_{y}^{2}+\omega) \partial_{\omega} f(\omega, a) 
+
f(\omega, a)
.
\end{align*}
\end{lemma}



\begin{lemma}\label{lem-d-v1}
Assume $p>1$, and let $0< \varepsilon < \varepsilon_{1}$. 
 Then, the following hold in $L^{2}(\R \times \T)$ for all $(\omega, a) \in (\omega_{p}-\delta_{0}, \omega_{p}+\delta_{0})\times (-a_{\varepsilon}, a_{\varepsilon})$:
\begin{align}
\label{phi-1}
\partial_{a}
V_{1}(\omega, a)
&=
V_{2}(\omega, a) \partial_{a}\varphi(\omega, a)
, 
\\[6pt]
\label{phi-1-2}
\partial_{\omega}
V_{1}(\omega, a)
&=
V_{2}(\omega, a) \partial_{\omega}\varphi(\omega, a)
.
\end{align} 
\end{lemma}
\begin{proof}[Proof of Lemma \ref{lem-d-v1}]
We shall prove \eqref{phi-1}. It follows from \eqref{v-k-deriv} that 
\begin{equation}\label{9/2-15:1}
\partial_{a} V_{1}(\omega, a)
=
V_{2}(\omega, a) \partial_{a} \varphi(\omega, a)
\quad 
\mbox{everywhere in $\R \times \T$}.
\end{equation}
We may assume $\sqrt{\omega} >\dfrac{1}{2}\sqrt{\omega_{p}} \gg \varepsilon_{1} >\varepsilon $.
 Then, by the fundamental theorem of calculus, \eqref{9/7-9:32} and \eqref{9/7-9:33}, we see that 
\begin{equation}\label{9/2-15:15}
\begin{split} 
\Big|
\frac{V_{1}(\omega, a+\delta)
-
V_{1}(\omega, a)
}{\delta}
\Big|
&\le 
\int_{0}^{1} V_{2}(\omega, a+\theta \delta)
|\partial_{a}\varphi(\omega, a+\theta \delta) |
\,d\theta 
\\[6pt]
&\lesssim 
e^{-\frac{p-1}{2}\sqrt{\omega_{p}}|x| 
+ \varepsilon (4+p)|x|}
\quad 
\mbox{everywhere in $\R \times \T$}, 
\end{split} 
\end{equation}
where the implicit constant depends only on $p$ and $\varepsilon$.
Hence, Lebesgue's dominated convergence theorem together with \eqref{9/2-15:1} 
 shows that 
\begin{equation}\label{9/2-15:13}
\partial_{a} V_{1}(\omega, a)
=
V_{2}(\omega, a)
\partial_{a}\varphi(\omega, a)
\quad 
\mbox{in $L^{1}(\R \times \T)$}
.
\end{equation}
Furthermore, by \eqref{9/2-15:15}, \eqref{9/7-9:32}, and \eqref{9/7-9:33}, we see that
\begin{equation}\label{9/2-15:17}
\limsup_{\delta \to 0}
\|
\frac{V_{1}(\omega, a+\delta) -V_{1}(\omega, a)}{\delta}
-
V_{2}(\omega, a)
\partial_{a}\varphi(\omega, a)
\|_{L^{\infty}(\R \times \T)}
\lesssim 1, 
\end{equation}
where the implicit constant depends only on $p$, $\varepsilon$ and $k$. 
Putting \eqref{9/2-15:13} and \eqref{9/2-15:17} together, we find that 
\eqref{phi-1} holds. Similarly, we can prove \eqref{phi-1-2}. 
\end{proof}


By \eqref{9/7-9:32}, \eqref{9/7-9:33}, 
Lemma \ref{lem-d-v1} 
and a direct computation, we can obtain the following lemma: 
\begin{lemma}\label{d-pot}
Assume $p>1$ and let $0<\varepsilon < \varepsilon_{1}$. 
If $ f \colon (\omega_{p}-\delta_{0}, \omega_{p}+\delta_{0}) \times (-a_{\varepsilon}, a_{\varepsilon}) \to H^{2}(\R \times \T)$ has the 
partial derivatives in $H^{2}(\R \times \T)$ ($\partial_{a}f(\omega, a), \partial_{\omega}f(\omega, a)\in H^{2}(\R \times \T)$), then the following holds in $L^{2}(\R \times \T)$ for all $(\omega, a) \in (\omega_{p}-\delta_{0}, \omega_{p}+\delta_{0}) \times 
(-a_{\varepsilon}, a_{\varepsilon})$:
\begin{align*}
\partial_{a}
\{ V_{1}(\omega, a)f(\omega, a) \} 
&=
V_{1}(\omega, a) \partial_{a}f(\omega, a)
+
V_{2}(\omega, a) \partial_{a} \varphi(\omega, a)f(\omega, a) , 
\\[6pt]
\partial_{\omega}
\{ V_{1}(\omega, a) f(\omega, a) \} 
&=
V_{1}(\omega, a) \partial_{\omega}f(\omega, a)
+
V_{2}(\omega, a)\partial_{\omega} \varphi(\omega, a)f(\omega, a) 
.
\end{align*}
\end{lemma}
\begin{remark}
Observe that we assume the differentiability of $f(\omega, a)$ in $H^{2}(\R \times \T)$, but 
the derivatives of the product $V_{1}(\omega, a)f(\omega, a)$ is taken in the $L^{2}(\R \times \T)$-sense. 
\end{remark}




The following lemma immediately follows 
from Lemmas \ref{d-lap} and \ref{d-pot}: 
\begin{lemma}\label{d-L}
Assume $p>1$ and let $0<\varepsilon < \varepsilon_{1}$. 
 If $f \colon (\omega_{p}-\delta_{0}, \omega_{p}+\delta_{0}) \times (-a_{\varepsilon}, a_{\varepsilon}) \to H^{2}(\R \times \T)$ has the partial derivatives 
in $H^{2}(\R \times \T)$ ($\partial_{a}f (\omega, a), \partial_{\omega}f(\omega, a) \in H^{2}(\R \times \T)$), then, 
the following hold in $L^{2}(\R \times \T)$ for all $(\omega, a) \in (\omega_{p}-\delta_{0}, \omega_{p}+\delta_{0})$:
\begin{align*}
\partial_{a} \{ 
\mathbf{L}_{+}(\omega, a) f(\omega, a)
\}
&=
\mathbf{L}_{+}(\omega, a) \partial_{a} f(\omega, a)
-
V_{2}(\omega, a) \partial_{a}\varphi(\omega, a) f(\omega, a)
, 
\\[6pt]
\partial_{\omega} 
\{
\mathbf{L}_{+}(\omega, a) f(\omega, a)
\} 
&=
\mathbf{L}_{+}(\omega, a) 
\partial_{\omega} f(\omega, a)
+
f(\omega, a)
-
V_{2}(\omega, a) \partial_{\omega}\varphi(\omega, a) f(\omega, a)
.
\end{align*}
\end{lemma}


\begin{lemma}\label{d-inverse-L}
Assume $p>1$ and $0< \varepsilon <\varepsilon_{1}$. If $f \colon (\omega_{p}-\delta_{0}, \omega_{p}+\delta_{0}) \times (-a_{\varepsilon}, a_{\varepsilon}) \to Y_{2}$ has the partial derivatives in $Y_{2}$ ($\partial_{a}f(\omega, a), \partial_{\omega}f(\omega, a) \in Y_{2}$), then the following hold in $H^{2}(\R \times \T)$ for all $(\omega, a) \in (\omega_{p}-\delta_{0}, \omega_{p}+\delta_{0}) \times (-a_{\varepsilon}, a_{\varepsilon})$: If 
\begin{equation*}\label{9/11-16:25}
\widetilde{f}(\omega, a):= \mathbf{T}(\omega, a)^{-1} f(\omega, a)
, 
\end{equation*}
then 
\begin{align}
\label{identity-3}
\partial_{a} \widetilde{f}(\omega, a)
&=
\mathbf{T}(\omega, a)^{-1} \partial_{a} f(\omega, a)
+
\mathbf{T}(\omega, a)^{-1} 
P_{\perp}\{V_{2}(\omega, a) \partial_{a}\varphi(\omega, a) 
\widetilde{f}(\omega, a)
\}
, 
\\[6pt]
\label{identity-4}
\partial_{\omega} \widetilde{f}(\omega, a)
&=
\mathbf{T}(\omega, a)^{-1} \partial_{\omega} f(\omega, a)
-
\mathbf{T}(\omega, a)^{-1} 
\widetilde{f}(\omega, a)
+
\mathbf{T}(\omega, a)^{-1} 
P_{\perp}
\big\{
V_{2}(\omega, a) \partial_{\omega}\varphi(\omega, a)
\widetilde{f}(\omega, a)
\big\}
. 
\end{align} 
\end{lemma}
\begin{proof}[Proof of Lemma \ref{d-inverse-L}]
We shall prove \eqref{identity-3}. 
We compute the derivative of $\widetilde{f}(\omega, a):=\mathbf{T}(\omega, a)^{-1} f(\omega, a)$ with respect to $a$ in accordance with the definition. Let $(\omega, a ) \in (\omega_{p}-\delta_{0}, \omega_{p}+\delta_{0}) \times (-a_{\varepsilon}, a_{\varepsilon})$, and 
let $\delta$ be a constant to be taken $\delta \to 0$. 
It is easy to see that 
the following identity holds everywhere in $\R \times \T$: 
\begin{equation} \label{lem-p1}
\begin{split}
&
\frac{\widetilde{f}(\omega, a+\delta)-\widetilde{f}(\omega, a)}{\delta}
=
\frac{ \mathbf{T}(\omega, a + \delta)^{-1} 
f(\omega, a+\delta)
- 
\mathbf{T}_{+} (\omega, a)^{-1} 
f(\omega, a)}{\delta}
\\[6pt]
&=
\delta^{-1}\{ 
\mathbf{T} (\omega, a + \delta)^{-1} 
- 
\mathbf{T}(\omega, a)^{-1}
\big\} 
f(\omega, a+\delta )
\\[6pt]
&\quad +
\mathbf{T}(\omega, a)^{-1} 
\frac{f(\omega, a+\delta) - f(\omega, a)}{\delta}
.
\end{split}
\end{equation}
Using \eqref{29-2} and \eqref{29-1}, we can rewrite 
the first term on the right-hand side of \eqref{lem-p1} as follows: 
\begin{equation}\label{lem-p4}
\begin{split} 
&
\delta^{-1}\{ 
\mathbf{T}(\omega, a + \delta)^{-1} 
- 
\mathbf{T}(\omega, a)^{-1}
\big\} 
f(\omega, a+\delta)
\\[6pt]
&=
-
\delta^{-1}
\mathbf{T}(\omega, a)^{-1} 
\big\{ 
\mathbf{T}(\omega, a+\delta)
-
\mathbf{T}(\omega, a)
\big\}
\mathbf{T}(\omega, a+\delta)^{-1}
f(\omega, a+\delta)
\\[6pt]
&= 
\mathbf{T}(\omega, a)^{-1} 
P_{\perp}
\Big( 
\frac{ V_{1}(\omega, a+\delta)
-
V_{1} (\omega, a)
}{\delta}
\mathbf{T}(\omega, a+\delta)^{-1}
f(\omega, a+\delta)
\Big)
\\[6pt]
&= 
\mathbf{T}(\omega, a)^{-1} 
P_{\perp}
\Big(
\frac{ V_{1}(\omega, a+\delta)
-
V_{1}(\omega, a)
}{\delta}
\mathbf{T}(\omega, a)^{-1}
f(\omega, a)
\Big)
\\[6pt]
&\quad 
+ 
\mathbf{T}(\omega, a)^{-1} 
P_{\perp}
\Big( 
\{V_{1}(\omega, a+\delta)
-
V_{1}(\omega, a)
\big\}
\mathbf{T}(\omega, a)^{-1}
\frac{f(\omega, a+\delta)
-
f(\omega, a)}{\delta}
\Big)
\\[6pt]
&\quad 
+
\mathbf{T}(\omega, a)^{-1} 
P_{\perp}
\Big( 
\frac{ 
V_{1}(\omega, a+\delta)
-
V_{1}(\omega, a)
}{\delta}
\big\{
\mathbf{T}(\omega, a+\delta)^{-1}
-
\mathbf{T}(\omega, a)^{-1}
\big\}
f(\omega, a+\delta)
\Big)
.
\end{split}
\end{equation}

We consider the first term on the right-hand side of \eqref{lem-p4}. 
It follows from the continuity of $\mathbf{T}(\omega, a)^{-1} \colon L^{2}(\R \times \T) \to H^{2}(\R \times \T)$ (see \eqref{28-4}) and the differentiability of $V_{1}(\omega, a)$ in $L^{2}(\R \times \T)$ (see Lemma \ref{lem-d-v1}) that 
\begin{equation}\label{28-21}
\begin{split}
&
\lim_{\delta \to 0}
\mathbf{T}(\omega, a)^{-1} 
P_{\perp}
\Big( 
\frac{ 
V_{1}(\omega, a+\delta)
-
V_{1}(\omega, a)
}{\delta}
\mathbf{T}(\omega, a)^{-1}
f(\omega, a)
\Big)
\\[6pt]
&=
\mathbf{T}(\omega, a)^{-1} 
P_{\perp} \{ V_{2}(\omega, a) \partial_{a}\varphi(\omega, a)\mathbf{T}(\omega, a)^{-1}
f(\omega, a)\}
\quad 
\mbox{in $H^{2}(\R \times \T)$}
.
\end{split}
\end{equation}

Next, we consider the second term on the right-hand side of \eqref{lem-p4}. 
Let $1 < p \leq 2$
By \eqref{28-3}, a convexity ($0<p-1<1$ 
), the continuity of $\varphi(\omega, a)$ with respect to $a$ in $H^{2}(\R \times \T)$, and the differentiability of 
$f(\omega, a)$, we see that 
\begin{equation}\label{28-22}
\begin{split}
&
\|
\mathbf{T}(\omega, a)^{-1} 
P_{\perp}
\Big( 
\{ V_{1}(\omega, a+\delta)
-
V_{1}(\omega, a)
\}
\mathbf{T}(\omega, a)^{-1}
\frac{f(\omega, a+\delta)
-
f(\omega, a)}{\delta}
\Big)
\|_{H^{2}_{x, y}(\R \times \T)}
\\[6pt]
&\lesssim 
\|
\{ 
V_{1}(\omega, a+\delta)
-
V_{1}(\omega, a)
\}
\mathbf{T}(\omega, a)^{-1}
\frac{f(\omega, a+\delta)
-
f(\omega, a)}{\delta}
\|_{L^{2}_{x, y}(\R \times \T)}
\\[6pt]
&\le 
\| \varphi(\omega, a+\delta)
-
\varphi(\omega, a)
\|_{L^{\infty}(\R \times \T)}^{p-1}
\Big\| 
\frac{f(\omega, a+\delta)
-
f(\omega, a)}{\delta}
\Big\|_{L^{2}_{x, y}(\R \times \T)}
\\[6pt]
&\to 0
\quad \mbox{as $\delta \to 0$}
.
\end{split}
\end{equation}
We can prove the case of 
$p \geq 2$ similarly. 

We consider the last term on the right-hand side of \eqref{lem-p4}. 
By the fundamental theorem of calculus, Lemma \ref{lem-d-v1}, 
\eqref{9/7-9:32}, and \eqref{9/7-9:33}, 
we see that 
\begin{equation}\label{29-5}
\|
\frac{ V_{1}(\omega, a+\delta) - V_{1}(\omega, a)}{\delta}
\|_{L^{\infty}(\R \times \T)}
=
\|
\int_{0}^{1} V_{2}(\omega, a+\theta \delta) \partial_{a}\varphi (\omega, a+\theta \delta) 
\,d \theta 
\|_{L^{\infty}(\R \times \T)}
\lesssim 1,
\end{equation}
where the implicit constant depends only on $p$ and $\varepsilon$. 
Then, by \eqref{28-3}, \eqref{29-5} and Lemma \ref{c-L}, we see that 
\begin{equation}\label{28-23}
\begin{split} 
&
\|\mathbf{T}(\omega, a)^{-1} 
P_{\perp}
\Big( 
\frac{ 
V_{1}(\omega, a+\delta)
-
V_{1}(\omega, a)
}{\delta}
\{
\mathbf{T}(\omega, a+\delta)^{-1}
-
\mathbf{T}(\omega, a)^{-1}
\}
f(\omega, a+\delta)
\Big)
\|_{H_{x, y}^{2}(\R \times \T)}
\\[6pt]
&\lesssim 
\|
\frac{ 
V_{1}(\omega, a+\delta)
-
V_{1}(\omega, a)
}{\delta}
\{
\mathbf{T}(\omega, a+\delta)^{-1}
-
\mathbf{T}(\omega, a)^{-1}
\}
f(\omega, a+\delta)
\|_{L^{2}(\R \times \T)}
\\[6pt]
&\to 0
\quad 
\mbox{as $\delta \to 0$}
.
\end{split}
\end{equation}

It remains to consider the second term on the right-hand side of \eqref{lem-p1}.
By the continuity of $\mathbf{T}(\omega, a )^{-1} \colon L^{2}(\R \times \T) \to H^{2}(\R \times \T)$ (see \eqref{28-4}) and the differentiability of $f(\omega, a)$, we see that 
\begin{equation}\label{lem-p2}
\lim_{\delta \to 0}
\mathbf{T}(\omega, a )^{-1}
\frac{f(\omega, a+\delta) - f(\omega, a)}{\delta}
=
\mathbf{T}(\omega, a )^{-1} \partial_{a}f(\omega, a)
\quad 
\mbox{in $H^{2}(\R \times \T)$}.
\end{equation}

Putting the above computations together, we find that \eqref{identity-3} holds. Similarly, we can prove \eqref{identity-4}. 

\end{proof}


\subsection{Second derivatives of $\boldsymbol{\varphi(\omega, a)}$}
\label{sec-d2-phi}

In this section, we compute the second derivatives of $\varphi(\omega, a)$ with respect to $\omega$ and $a$. We emphasize that when $1<p<2$, the twice differentiability of $\varphi(\omega, a)$ is not obvious, as the nonlinearity of the equation \eqref{sp} is not $C^{2}$.

Recall that $\eta \colon (\omega_{p} - \delta_{0}, \omega_{p} 
+ \delta_{0}) \times (-a_{0}, a_{0}) \to X_{2}$ is $C^{1}$ 
and $P_{\perp} \mathcal{F}(\omega, \varphi(\omega, a))=0$ (see Lemma \ref{main-lem-2}). 
Observe from a direct computation that the following holds in $L^{2}(\R \times \T)$: 
\begin{align}
\label{d-phi-p-a}
\partial_{a} \{
V_{0}(\omega, a) \}
&=
p \varphi (\omega, a)^{p-1} \partial_{a} \varphi(\omega, a) 
=
V_{1}(\omega, a) \{ \psi_{\omega_{p}}\cos{y}+ \partial_{a} h(\omega, a)\}
, 
\\[6pt]
\label{d-phi-p-w}
\partial_{\omega}
\{V_{0}(\omega, a)\}
&=
p \varphi(\omega, a)^{p-1} \partial_{\omega} \varphi(\omega, a)
=
V_{1}(\omega, a) \partial_{\omega} h(\omega, a) . 
\end{align}
Note that Lemma \ref{main-lem-2} shows that  
\begin{align}
\label{der-w-h} 
\partial_{\omega} \varphi(\omega, a)
&=-\mathbf{T}(\omega, a)^{-1}
\{ R_{\omega_{p}} + \eta(\omega, a) \}
,
\\[6pt]
\label{der-a-h}
\partial_{a}\varphi(\omega, a)
&=
\psi_{\omega_{p}} \cos{y}
- \mathbf{T}(\omega, a)^{-1}P_\perp
\mathbf{L}_{+}(\omega, a)
\psi_{\omega_{p}} \cos{y}
. 
\end{align}


In order to prove the continuity of the second and higher derivatives of $\varphi(\omega, a)$ in $H^{2}(\R \times \T)$, 
we prepare the following lemma: 
\begin{lemma}\label{20/10/21/15:57}
Assume $p>1$ 
and 
let $k, j$ be integers satisfying 
$k > p$ and 
$j \geq 1$.
Then, there exists $0< \varepsilon(k,j)<\varepsilon_{1}$ depending only on $p$, $k$, and $j$ with the following property: 
Let $0<\varepsilon < \varepsilon(k,j)$, 
and let $g$ be a function in $C( (\omega_{p}-\delta_{0}, \omega_{p}+\delta_{0}) \times (-a_{\varepsilon}, a_{\varepsilon}), 
L^{2}(\R \times \T))$. 
Assume that for any $(\omega, a)\in (\omega_{p}-\delta_{0}, \omega_{p}+\delta_{0})\times (-a_{\varepsilon}, a_{\varepsilon})$ and $(x, y)\in \R \times \T$: 
\begin{equation}\label{20/10/29/14:59}
|g(\omega, a, x, y)|
\lesssim 
e^{- k \sqrt{\omega}|x| + j \varepsilon |x|}
, 
\end{equation}
where the implicit constant depends only on $p$, $k$, $j$ and $\varepsilon$. Furthermore, assume that 
\begin{equation}\label{20/10/29/16:37}
\lim_{(\gamma_{1}, \gamma_{2})\to (0,0)} 
\| 
e^{(k-1) \sqrt{\omega}|x| + k \varepsilon |x| }
\big\{
g(\omega+\gamma_{1},a+\gamma_{2})- g(\omega, a)
\big\} 
\|_{L^{2}(\R \times \T)}
= 0.
\end{equation}
Then, $\mathbf{T}(\omega, a)^{-1}P_{\perp}
\big\{ 
V_{k}(\omega, a)g(\omega, a) 
\big\}
$ is continuous with respect to $(\omega, a)$ on $(\omega_{p}-\delta_{0}, \omega_{p}+\delta_{0}) \times (-a_{\varepsilon}, a_{\varepsilon})$ in the $H^{2}(\R \times \T)$-topology.
\end{lemma}

\begin{proof}[Proof of Lemma \ref{20/10/21/15:57}]
Let $0< \varepsilon <\varepsilon_{1}$,  
  $(\omega, a) \in (\omega_{p}-\delta_{0}, \omega_{p}+\delta_{0}) \times (-a_{\varepsilon}, a_{\varepsilon})$, and let $(\gamma_{1}, \gamma_{2}) \in \mathbb{R}^{2}$. We will assume $\sqrt{\omega}>\dfrac{1}{2}\sqrt{\omega_{p}}$ and 
  $\varepsilon$ being sufficiently small dependently only on $p$, $k$, and $j$, without any notice. Furthermore,  we will take $(\gamma_{1}, \gamma_{2}) \to (0,0)$, so that we may assume that 
\begin{align}
\label{20/10/27/17:5}
&(\omega+\gamma_{1}, a+\gamma_{2}) 
\in 
(\omega_{p}-\delta_{0}, \omega_{p}+\delta_{0}) \times (-a_{\varepsilon}, a_{\varepsilon}), 
\\[6pt]
\label{20/10/29/17:56}
&
\sqrt{\omega}-\varepsilon
\le 
\sqrt{\omega-|\gamma_{1}|} 
\le 
\sqrt{\omega+|\gamma_{1}|} 
\le 
\sqrt{\omega}+\varepsilon
.
\end{align}


First, we shall show that 
\begin{equation}\label{20/10/26/18:15}
\lim_{(\gamma_{1}, \gamma_{2})\to (0,0)}
\big\| 
V_{k}(\omega+\gamma_{1}, a+\gamma_{2})
g(\omega+\gamma_{1}, a+\gamma_{2}) 
-
V_{k}(\omega, a) g(\omega, a) 
\big\|_{L^{2}(\R \times \T)}
=
0
. 
\end{equation} 
Observe that 
\begin{equation}\label{20/10/29/16:11}
\begin{split}
&
\big\| 
V_{k}(\omega+\gamma_{1}, a+\gamma_{2})
g(\omega+\gamma_{1}, a+\gamma_{2}) 
-
V_{k}(\omega, a) g(\omega, a) 
\big\|_{L^{2}(\R \times \T)}
\\[6pt]
&\le 
\| 
\big\{ 
V_{k}(\omega+\gamma_{1},a+\gamma_{2})
-
V_{k}(\omega, a)
\big\}
g(\omega+\gamma_{1}, a+\gamma_{2}) 
\|_{L^{2}(\R \times \T)}
\\[6pt]
&\quad + 
\|V_{k}(\omega, a) 
\big\{
g(\omega+\gamma_{1}, a+\gamma_{2}) 
-
g(\omega, a) 
\big\} 
\|_{L^{2}(\R \times \T)}. 
\end{split} 
\end{equation} 
Consider the first term on the right-hand side of \eqref{20/10/29/16:11}. 
By the fundamental theorem of calculus, \eqref{v-k-deriv}, 
\eqref{20/10/29/14:59}, 
\eqref{9/7-9:32}, \eqref{9/7-9:33}, \eqref{20/10/29/17:56} and  
 $p > 1$, 
 the following holds everywhere in $\R \times \T$: 
\begin{equation}\label{eq5-11}
\begin{split}
&
\big|
\big\{ 
V_{k}(\omega+\gamma_{1},a+\gamma_{2})
-
V_{k}(\omega, a)
\big\}
g(\omega+\gamma_{1}, a+\gamma_{2}) 
\big| 
\\[6pt]
&\lesssim 
\big| 
V_{k}(\omega+\gamma_{1},a+\gamma_{2})
-
V_{k}(\omega, a+\gamma_{2})
\big|
e^{-k 
\sqrt{\omega-|\gamma_{1}|}|x| +j \varepsilon |x|}
\\[6pt]
&\quad +
\big|
V_{k}(\omega, a+\gamma_{2})
-
V_{k}(\omega, a)
\big|
e^{-k 
\sqrt{\omega-|\gamma_{1}|}|x| +j \varepsilon |x|}
\\[6pt]
&\lesssim 
|\gamma_{1}|
\int_{0}^{1} \big| V_{k+1}(\omega+\theta \gamma_{1}, a+\gamma_{2}) \big|
\big| 
\partial_{\omega}\varphi(\omega+\theta \gamma_{1}, a+\gamma_{2})
\big| 
\,d\theta
e^{-k 
\sqrt{\omega-|\gamma_{1}|}|x| +j \varepsilon |x|}
\\[6pt]
&\quad +
|\gamma_{2}|
\int_{0}^{1} 
\big| V_{k+1}(\omega, a+\theta \gamma_{2}) \big| 
\big| 
\partial_{a}\varphi(\omega, a+\theta \gamma_{2})
\big| 
\,d\theta
e^{-k 
\sqrt{\omega-|\gamma_{1}|}|x| +j \varepsilon |x|}
\\[6pt]
&\lesssim
(|\gamma_{1}|+|\gamma_{2}|)
e^{-\frac{p}{2}\sqrt{\omega_{p}}|x| + (3k + j + 5 - 2p)\varepsilon |x|}
\lesssim 
(|\gamma_{1}|+|\gamma_{2}|)
e^{\frac{-p}{4}\sqrt{\omega_{p}}|x|}
, 
\end{split}
\end{equation}
where the implicit constants depend only on $p$, $k$, $j$ and $\varepsilon$. 
Thus, we find that 
\begin{equation}\label{eq5-12}
\lim_{(\gamma_{1}, \gamma_{2})\to (0,0)} 
\| 
\big\{ 
V_{k}(\omega+\gamma_{1},a+\gamma_{2})
-
V_{k}(\omega, a)
\big\}
g(\omega+\gamma_{1}, a+\gamma_{2}) 
\|_{L^{2}(\R \times \T)}
= 0. 
\end{equation}
Move on to the second term on the right-hand side of \eqref{20/10/29/16:11}. 
 By \eqref{9/7-9:32} and $p > 1$, we see  that 
\begin{equation}\label{eq5-13}
\big|
V_{k}(\omega, a) 
e^{-(k-1)\sqrt{\omega}|x| - k\varepsilon |x|}
\big| 
\lesssim  
e^{\frac{-(p-1)}{2}\sqrt{\omega_{p}}|x|}
, 
\end{equation}
where the implicit constant depends only on $p$, $k$, $j$ and $\varepsilon$. 
Then, by \eqref{eq5-13} and \eqref{20/10/29/16:37}, 
we see that 
\begin{equation}\label{eq5-14}
\lim_{(\gamma_{1}, \gamma_{2})\to (0,0)}
\|V_{k}(\omega, a) 
\big\{
g(\omega+\gamma_{1}, a+\gamma_{2}) 
-
g(\omega, a) 
\big\} 
\|_{L^{2}(\R \times \T)}
=0
.
\end{equation}
Putting \eqref{20/10/29/16:11}, \eqref{eq5-12} and \eqref{eq5-14} together, we find that \eqref{20/10/26/18:15} holds. 

We shall finish the proof of the lemma. 

Observe from 
\eqref{9/7-9:32}, 
\eqref{20/10/29/14:59}, 
$p > 1$ and  \eqref{20/10/29/17:56} that 
\begin{equation}\label{eq5-15}
\begin{split}
&
\|
V_{k}(\omega+\gamma_{1}, a+\gamma_{2}) 
g(\omega+\gamma_{1}, a+\gamma_{2}) 
\|_{L^{2}(\R \times \T)}
\\[6pt]
&\lesssim 
\|
e^{\frac{- p}{2} 
\sqrt{\omega_{p}} |x|
+ 
(3k - 2p) \varepsilon |x|}
\|_{L^{2}(\R \times \T)}
\le 
\|
e^{\frac{- p}{4}
\sqrt{\omega_{p}}|x|}
\|_{L^{2}(\R \times \T)}
\lesssim 1, 
\end{split} 
\end{equation}
where the implicit constants depend only on $p$, $k$, $j$ and $\varepsilon$. 
 Then, by \eqref{eq5-15}, Lemma \ref{c-L}, \eqref{28-3} and \eqref{20/10/26/18:15}, we can prove the continuity of  
$\mathbf{T}(\omega, a)^{-1}P_{\perp}
\big\{ 
V_{k}(\omega, a)g(\omega, a) 
\big\}
$ on $(\omega_{p}-\delta_{0}, \omega_{p}+\delta_{0}) \times (-a_{\varepsilon}, a_{\varepsilon})$ in the $H^{2}(\R \times \T)$-topology.
\end{proof}


Next, we give the second derivatives of $\varphi(\omega, a)$:
\begin{proposition}\label{main-lem-1}
Assume $p>1$. Then, there exists $0< \varepsilon_{2}<\varepsilon_{1}$ depending only on $p$ such that if $0< \varepsilon <\varepsilon_{2}$, 
 then $\varphi$ is $C^{2}$ 
on $(\omega_{p} - \delta_{0},~\omega_{p}+ \delta_{0}) 
\times (-a_{\varepsilon}, a_{\varepsilon})$ 
 in the $H^{2}(\R \times \T)$-topology; 
 and the following hold for all $(\omega, a) 
\in (\omega_{p} - \delta_{0},~\omega_{p}+ \delta_{0}) \times 
(-a_{\varepsilon}, a_{\varepsilon})$:
\begin{align}
\label{der2-h-a}
\partial_{a}^{2} \varphi (\omega, a)
&=
\mathbf{T}(\omega, a)^{-1}
P_{\perp}
\big(
V_{2}(\omega, a)
\{\partial_{a} \varphi(\omega, a) \}^{2}
\big)
, 
\\[6pt]
\label{der2-h-w}
\partial_{\omega}^{2} \varphi (\omega, a) 
&=
\mathbf{T}(\omega, a)^{-1}
P_{\perp}
\big( 
V_{2}(\omega, a) \{ \partial_{\omega} \varphi(\omega, a)\}^{2}
-
\partial_{\omega}\varphi(\omega, a)
\big)
,
\end{align}
\begin{equation}\label{der2-h-a-w}
\begin{split}
&\partial_{\omega} \partial_{a} 
\varphi(\omega, a)
=
\partial_{a} \partial_{\omega} 
\varphi(\omega, a)
\\[6pt]
&=
\mathbf{T}(\omega, a)^{-1}
P_{\perp}
\big(
V_{2}(\omega, a) \partial_{\omega}
\varphi(\omega, a) \partial_{a} 
\varphi(\omega, a) 
-
\partial_{a}\varphi(\omega, a)
\big)
.
\end{split}
\end{equation}
\end{proposition}
\begin{remark}\label{rem-2d-phi}
Since $\mathbf{T}(\omega, a)^{-1}$ maps $Y_{2}$ to $X_{2}$, Proposition \ref{main-lem-1} shows that 
\begin{equation*}\label{9/50-15:38}
\partial_{a}^{2} \varphi(\omega, a), ~ 
\partial_{\omega} \partial_{a} \varphi(\omega, a), ~
\partial_{a} \partial_{\omega} \varphi(\omega, a), ~
\partial_{\omega}^{2} \varphi(\omega, a) 
\in X_{2}.
\end{equation*}
\end{remark}

\begin{proof}[Proof of Proposition \ref{main-lem-1}]
We shall prove \eqref{der2-h-a} . 
By Lemmas \ref{p-d} and \ref{d-L}, the following holds 
in $L^{2}(\R \times \T)$: 
\begin{align}
\label{pf-2d-13}
\partial_{a} \big( P_{\perp} \mathbf{L}_{+}(\omega, a) \{ \psi_{\omega_{p}} \cos{y} \} \big)
&=
- 
P_{\perp} \Big( V_{2}(\omega, a) \partial_{a}\varphi(\omega, a) \psi_{\omega_{p}} \cos{y} \Big) 
, 
\\[6pt]
\label{pf-2d-14}
\partial_{\omega} \big( P_{\perp} \mathbf{L}_{+}(\omega, a) \{ \psi_{\omega_{p}} \cos{y} \}
\big)
&=
P_{\perp} \{ \psi_{\omega_{p}} \cos{y} \}
-
P_{\perp} \Big( V_{2}(\omega, a) \partial_{\omega}\varphi(\omega, a) \psi_{\omega_{p}} \cos{y} \Big) 
.
\end{align}
Furthermore, by \eqref{der-a-h}, Lemma \ref{d-inverse-L} and \eqref{pf-2d-13}, 
the following holds in $H^{2}(\R \times \T)$: 
\begin{equation}\label{pf-2d-13-2}
\begin{split}
&\partial_{a}^{2} \varphi(\omega, a)
=
-\partial_{a} 
\Big( 
\mathbf{T}(\omega, a)^{-1} 
P_{\perp} \mathbf{L}_{+}(\omega, a) \{ \psi_{\omega_{p}} \cos{y} \}
\Big)
\\[6pt]
&=
-\mathbf{T}(\omega, a)^{-1} \partial_{a} \Big(
P_{\perp} \mathbf{L}_{+}(\omega, a) \{ \psi_{\omega_{p}} \cos{y} \}
\Big)
\\[6pt]
&\quad 
-\mathbf{T}(\omega, a)^{-1} P_{\perp} \Big( V_{2}(\omega, a) \partial_{a}\varphi(\omega, a)\mathbf{T}(\omega, a)^{-1} 
P_{\perp} \mathbf{L}_{+}(\omega, a) \{ \psi_{\omega_{p}} \cos{y} \}
\Big)
\\[6pt]
&=
\mathbf{T}(\omega, a)^{-1}P_{\perp} \Big( V_{2}(\omega, a) \partial_{a} \varphi(\omega, a) 
\Big\{\psi_{\omega_{p}} \cos{y}
-\mathbf{T}(\omega, a)^{-1} 
P_{\perp} \mathbf{L}_{+}(\omega, a) \{ \psi_{\omega_{p}} \cos{y} \}
\Big\}
\Big)
.
\end{split} 
\end{equation}
Plugging \eqref{der-a-h} into the right-hand side of \eqref{pf-2d-13-2}, we obtain \eqref{der2-h-a}. Similarly, we can prove \eqref{der2-h-a-w} and \eqref{der2-h-w}. 

It remains to prove the continuity of the second derivatives. 
When $2\le p$, 
we see from the implicit function 
theorem that 
$\varphi(\omega, a)$ is 
$C^{2}$ with respect to $a$ 
and $\omega$. 
Hence, we may assume that $2>p$. 
Observe from \eqref{9/7-9:33} that 
\begin{equation}\label{eq5-16}
\begin{split}
\big| 
\{ \partial_{a}\varphi(\omega, a, x, y)\}^{2} 
\big|
\lesssim 
e^{-2\sqrt{\omega}|x|+4\varepsilon |x|}
, 
\end{split} 
\end{equation}
where the implicit constant depends only on $p$ and $\varepsilon$. 
Furthermore, by \eqref{9/7-9:33}, the embedding 
$H^{2}(\R \times \T) \hookrightarrow L^{\infty}(\R \times \T)$, and the continuity of $\partial_{a}\varphi(\omega, a)$ in $H^{2}(\R \times \T)$, 
we see that
\begin{equation}\label{eq5-17}
\begin{split}
&
\| e^{\sqrt{\omega}|x|+2\varepsilon |x|} 
\big\{ ( \partial_{a}\varphi(\omega+\gamma_{1},a+\gamma_{2}))^{2} 
- (\partial_{a}\varphi(\omega, a, x, y))^{2} 
\big\}
\|_{L^{2}(\R \times \T)}
\\[6pt]
&\le 
\| e^{\sqrt{\omega}|x|+2\varepsilon |x|} 
e^{-\frac{3}{2}(\sqrt{\omega-|\gamma_{1}|} + 2\varepsilon)|x| }\|_{L^{2}(\R \times \T)}
\| 
\varphi(\omega+\gamma_{1},a+\gamma_{2}) 
- \partial_{a}\varphi(\omega, a, x, y) 
\|_{L^{\infty}(\R \times \T)}^{\frac{1}{2}}
\\[6pt]
&\to 0
\quad 
\mbox{as $(\gamma_{1}, \gamma_{2}) \to (0,0)$}.
\end{split} 
\end{equation} 
Then, by \eqref{der2-h-a}, \eqref{eq5-16} and \eqref{eq5-17}, 
we find that Lemma \ref{20/10/21/15:57} can apply to $\partial_{a}^{2}\varphi(\omega, a)$ 
as $k=2$, $j=4$, $g(\omega, a)=\{ \partial_{a}\varphi(\omega, a)\}^{2}$. Thus, we see that 
$\partial_{a}^{2}\varphi(\omega, a)$ is continuous with respect to 
$(\omega, a)$ in the 
$H^{2}(\R \times \T)$-topology. 
Similarly, we can prove the continuity 
of the other partial derivatives. 
\end{proof}

We state decay properties of the derivatives of $\varphi(\omega, a)$: 
\begin{lemma}\label{9/2-13:45}
Let $p>1$. Then, there exists $\widetilde{\varepsilon}_{2}>0$ depending only on $p$ such that 
for any $0<\varepsilon <\widetilde{\varepsilon}_{2}$ 
and $(\omega, a) \in (\omega_{p} - \delta_{0}, \omega_{p}+ \delta_{0}) \times (-a_{\varepsilon}, a_{\varepsilon})$, 
the following holds: 
\begin{equation}\label{dec2-1}
|\partial_{a}^{2} \varphi (\omega, a,x, y) |
+
|
\partial_{\omega} \partial_{a} \varphi(\omega, a,x, y)
|
+
| 
\partial_{\omega}^{2} \varphi(\omega, a,x, y) 
|
\lesssim 
e^{-\{ \sqrt{\omega} -3\varepsilon \}|x| }
, 
\end{equation}
where the implicit constants depend only on $p$ and $\varepsilon$. 
\end{lemma}
\begin{proof}
We may write \eqref{der2-h-a} through \eqref{der2-h-w} in Proposition \ref{main-lem-1} as follows: 
\begin{align}
\label{2d-eq-1}
\mathbf{L}_{+}(\omega, a) \partial_{a}^{2} \varphi(\omega, a)
&=
V_{2}(\omega, a)
\{\partial_{a} \varphi(\omega, a) \}^{2}, 
\\[6pt]
\label{2d-eq-2}
\mathbf{L}_{+}(\omega, a) \partial_{\omega}\partial_{a} \varphi(\omega, a)
&=
V_{2}(\omega, a) \partial_{\omega} \varphi(\omega, a) \partial_{a} \varphi(\omega, a) - \partial_{a} \varphi(\omega, a) , 
\\[6pt]
\label{2d-eq-4}
\mathbf{L}_{+}(\omega, a) \partial_{\omega}^{2} \varphi(\omega, a)
&=
V_{2}(\omega, a)
\{\partial_{\omega} \varphi(\omega, a) \}^{2}
-
\partial_{\omega} \varphi(\omega, a)
. 
\end{align}
Observe from \eqref{9/7-9:32} and \eqref{9/7-9:33} that 
if $\varepsilon$ is sufficiently small depending only on $p$, then 
\begin{align*}
|
V_{2}(\omega, a)
\{\partial_{a} \varphi(\omega, a) \}^{2}
|
&\lesssim
e^{- \{ \sqrt{\omega}- 2\varepsilon\} |x|}, 
\\[6pt]
|
V_{2}(\omega, a)
\partial_{\omega}\varphi(\omega, a)\partial_{a} \varphi(\omega, a) 
|
+
|\partial_{a}\varphi(\omega, a)|
&
\lesssim 
e^{- \{ \sqrt{\omega}- 2\varepsilon\} |x|}, 
\\[6pt]
|
V_{2}(\omega, a)
\{\partial_{\omega} \varphi(\omega, a) \}^{2}
|
+
|\partial_{\omega} \varphi(\omega, a)|
&\lesssim 
e^{- \{ \sqrt{\omega}- 2\varepsilon\} |x|}, 
\end{align*}
where the implicit constants depend only on $p$ and $\varepsilon$. 
Then, applying Proposition \ref{thm-exd-1} to \eqref{2d-eq-1} through \eqref{2d-eq-4} as $\alpha=\sqrt{\omega} -2\varepsilon$, we obtain \eqref{dec2-1}. 
\end{proof}


\subsection{Third derivatives of $\boldsymbol{\varphi(\omega, a)}$}
\label{sec-d3-phi}

In this section, we find the third derivatives of $\varphi(\omega,a)$ with respect to $\omega$ and $a$. 


Let us begin with a generalization of Lemma \ref{lem-d-v1}:

\begin{lemma}\label{lem-phi-v2}
Assume $p>1$, and 
let $k, j$ be integers 
satisfying $k > p$ and $j \ge 1$.
Then, there exists $\varepsilon(k,j)>0$ depending only on $p$, $k$, and $j$ with the following property: 
Let $0<\varepsilon < \varepsilon(k,j)$, 
and let $g \in C^{1}((\omega_{p}-\delta_{0}, \omega_{p}+\delta_{0}) \times (-a_{\varepsilon}, a_{\varepsilon}), L^{\infty}(\R \times \T) )$.
Furthermore, assume that the following hold for all $(\omega, a)\in (\omega_{p}-\delta_{0}, \omega_{p}+\delta_{0})\times (-a_{\varepsilon}, a_{\varepsilon})$
and $(x, y)\in \R \times \T$:
\begin{align}
\label{eq5-18}
|g(\omega, a, x, y)|
&\lesssim 
e^{- k \sqrt{\omega}|x|+ j \varepsilon |x|}
, 
\\[6pt]
\label{eq5-19}
|\partial_{\omega} g(\omega, a, x, y)|
+
|\partial_{a} g(\omega, a, x, y)|
&\lesssim 
e^{- k \sqrt{\omega}|x|+ (j+2) \varepsilon |x|
}
, 
\end{align}
where the implicit constant depends only on $p$, 
$j$ and $\varepsilon$. Then, the following holds in the $L^{2}(\R \times \T)$-topology for all $(\omega, a) \in (\omega_{p}-\delta_{0}, \omega_{p}+\delta_{0})\times (-a_{\varepsilon}, a_{\varepsilon})$
: 
\begin{equation*}\label{20/11/9/11:11}
\partial 
\big\{ 
V_{k}(\omega, a) 
g(\omega, a) 
\big\}
=
V_{k+1}(\omega, a) 
\partial \varphi(\omega, a)
g(\omega, a)
+
V_{k}(\omega, a) 
\partial g(\omega, a), 
\end{equation*} 
where $\partial$ denotes 
either $\partial_{a}$ or $\partial_{\omega}$.
\end{lemma}
We omit the proof of Lemma \ref{lem-phi-v2} as the lemma can be proven in a way similar to Lemma \ref{lem-d-v1}.

Now, we give the third derivatives of $\varphi(\omega,a)$: 
\begin{proposition}
\label{main-lem-3}
Assume $p>1$ and let $\varepsilon_{2}$ be the constant give in Proposition \ref{main-lem-1}. 
 Then, there exists $0< \varepsilon_{3}<\varepsilon_{2}$ depending only on $p$ such that
if $0<\varepsilon < \varepsilon_{3}$, then 
$\varphi$ is of class $C^{3}$ on $(\omega_{p} - \delta_{0}, \omega_{p}+ \delta_{0}) \times (-a_{\varepsilon}, a_{\varepsilon})$ in the $H^{2}(\R \times \T)$-topology; 
and the following hold for all $(\omega, a) \in (\omega_{p} - \delta_{0}, \omega_{p}+ \delta_{0}) \times (-a_{\varepsilon}, a_{\varepsilon})$:

\begin{equation}\label{eq5-29}
\begin{split} 
\partial_{a}^{3} \varphi (\omega, a)
&=
\mathbf{T}(\omega, a)^{-1}
P_{\perp}
\Big(
V_{3}(\omega, a)
\{
\partial_{a}\varphi(\omega, a)
\}^{3}
\Big) 
\\[6pt]
&\quad + 3
\mathbf{T}(\omega, a)^{-1} 
P_{\perp} 
\Big(
V_{2}(\omega, a)
\partial_{a}\varphi(\omega, a)
\partial_{a}^{2} \varphi(\omega, a)
\Big), 
\end{split}
\end{equation} 
\begin{equation}\label{eq5-30}
\begin{split} 
&
\partial_{a}^{2} \partial_{\omega} \varphi(\omega, a)
=
\partial_{a} \partial_{\omega} \partial_{a} \varphi 
=
\partial_{\omega} \partial_{a}^{2} \varphi
\\[6pt]
&=
\mathbf{T}(\omega, a)^{-1}
P_{\perp}
\Big(
V_{3}(\omega, a) 
\partial_{\omega}\varphi(\omega, a)
\{
\partial_{a}\varphi(\omega, a)
\}^{2}
\Big) 
\\[6pt] 
&\quad +2
\mathbf{T}(\omega, a)^{-1}
P_{\perp}
\Big(
V_{2}(\omega, a) 
\partial_{a}\varphi(\omega, a) 
\partial_{a} \partial_{\omega} \varphi
(\omega, a)
\Big) 
\\[6pt] 
&\quad + 
\mathbf{T}(\omega, a)^{-1}
P_{\perp}
\Big( 
V_{2}(\omega, a) 
\partial_{a}^{2}\varphi(\omega, a)
\partial_{\omega}\varphi(\omega, a)
\Big) 
\\[6pt]
&\quad -\mathbf{T}(\omega, a)^{-1}
P_{\perp} 
\partial_{a}^{2}\varphi(\omega, a)
, 
\end{split} 
\end{equation}

\begin{equation}\label{eq5-33}
\begin{split} 
&
\partial_{a} \partial_{\omega}^{2} \varphi(\omega, a)
=
\partial_{\omega} \partial_{a}\partial_{\omega}\varphi 
=
\partial_{\omega}^{2} \partial_{a} \varphi
\\[6pt]
&=
\mathbf{T}(\omega, a)^{-1}
P_{\perp}
\Big(
V_{3}(\omega, a) 
\partial_{\omega}\varphi(\omega, a)
\{
\partial_{a}\varphi(\omega, a)
\}^{2}
\Big) 
\\[6pt]
&\quad +2
\mathbf{T}(\omega, a)^{-1}
P_{\perp}
\Big(
V_{2}(\omega, a) 
\partial_{a}\varphi(\omega, a) 
\partial_{a} \partial_{\omega} \varphi
(\omega, a)
\Big) 
\\[6pt]
&\quad + 
\mathbf{T}(\omega, a)^{-1}
P_{\perp}
\Big( 
V_{2}(\omega, a) 
\partial_{a}^{2}\varphi(\omega, a)
\partial_{\omega}\varphi(\omega, a)
\Big) 
\\[6pt]
& \quad 
-\mathbf{T}(\omega, a)^{-1}
P_{\perp} 
\partial_{a}^{2}\varphi(\omega, a)
, 
\end{split}
\end{equation}

\begin{equation}\label{eq5-32}
\begin{split} 
\partial_{\omega}^{3} \varphi (\omega, a)
&= 
\mathbf{T}(\omega, a)^{-1}
P_{\perp}
\Big(
V_{3}(\omega, a) 
\{
\partial_{\omega}\varphi(\omega, a)
\}^{3}
\Big) 
\\[6pt]
& \quad + 
3 \mathbf{T}(\omega, a)^{-1}
P_{\perp}
\Big( 
V_{2}(\omega, a) 
\partial_{\omega}^{2}\varphi(\omega, a)
\partial_{\omega}\varphi(\omega, a)
\Big) 
\\[6pt]
& \quad -w\mathbf{T}(\omega, a)^{-1}
P_{\perp} 
\partial_{\omega}^{2}\varphi(\omega, a)
.
\end{split}
\end{equation} 
\end{proposition}
\begin{remark}\label{rem-4}
Since $\mathbf{T}(\omega, a)^{-1}$ maps $Y_{2}$ to $X_{2}$, Proposition \ref{main-lem-3} shows that 
\begin{equation*}\label{20/10/31/16:8}
\partial_{a}^{3} \varphi(\omega, a), ~ 
\partial_{a}^{2} \partial_{\omega} \varphi(\omega, a), ~
\partial_{a} \partial_{\omega}^{2} \varphi(\omega, a), ~
\partial_{\omega}^{3} \varphi(\omega, a) 
\in X_{2}.
\end{equation*}
\end{remark}
\begin{proof}[Proof of Proposition 
\ref{main-lem-3}]
We shall prove \eqref{eq5-29}. By \eqref{der2-h-a} in Proposition \ref{main-lem-1}, 
 Lemma \ref{d-inverse-L} with $f=P_{\perp}(V_{2}(\omega, a)\{\partial_{a} \varphi(\omega, a) \}^{2})$, 
 Lemma \ref{p-d} and Lemma \ref{lem-phi-v2} with $k=2$, $j=4$, $g=\{\partial_{a}\varphi\}^{2}$, 
we see that the following holds in $H^{2}(\R \times \T)$: 
\begin{equation*}\label{9/1-16:39}
\begin{split}
&\partial_{a}^{3}\varphi(\omega, a)
=
\partial_{a}
\partial_{a}^{2}\varphi(\omega, a)
=
\partial_{a}
\big\{ 
\mathbf{T}(\omega, a)^{-1}
P_{\perp}
\big(
V_{2}(\omega, a)
\{\partial_{a} \varphi(\omega, a) \}^{2}
\big)
\big\}
\\[6pt]
&=
\mathbf{T}(\omega, a)^{-1}
P_{\perp}
\partial_{a} 
\big(
V_{2}(\omega, a)
\{\partial_{a} \varphi(\omega, a) \}^{2}
\big)
\\[6pt]
& \quad +
\mathbf{T}(\omega, a)^{-1}
P_{\perp} 
\big\{
V_{2}(\omega, a) \partial_{a} \varphi(\omega, a) 
\partial_{a}^{2}\varphi(\omega, a)
\big\}
\\[6pt]
&=
\mathbf{T}(\omega, a)^{-1}
P_{\perp}
\big(
V_{3}(\omega, a) \{ \partial_{a} \varphi(\omega, a)\}^{3}
+
2V_{2}(\omega, a) \partial_{a}\varphi(\omega, a)
\partial_{a}^{2}\varphi(\omega, a)
\big)
\\[6pt]
& \quad +
\mathbf{T}(\omega, a)^{-1}
P_{\perp} 
\big\{
V_{2}(\omega, a) \partial_{a} \varphi(\omega, a) 
\partial_{a}^{2} \varphi(\omega, a)
\big\}
.
\end{split} 
\end{equation*}
Thus, we have proved \eqref{eq5-29}. 
Similarly, we can prove \eqref{eq5-30} through \eqref{eq5-32}. 

It remains to prove the continuity of the third derivatives of $\varphi$. 
When $3 \le p$, 
we see from the implicit function 
theorem that 
$\varphi(\omega, a)$ is 
$C^{3}$ with respect to $a$ 
and $\omega$. 
Hence, we may assume that 
$p < 3$. Then, we shall prove the continuity of $\partial_{a}^{3}\varphi$. 
Observe from the continuity of $\partial_{a}\varphi$ and \eqref{9/7-9:33} 
that Lemma \ref{20/10/21/15:57} 
can apply to the first term on the right-hand side of \eqref{eq5-29} 
as $k=3$, $j=6$, $g(\omega, a)=\{ \partial_{a}\varphi(\omega, a)\}^{3}$. 
Moreover, observe from the continuity of $\partial_{a}\varphi$ and $\partial_{a}^{2}\varphi$, \eqref{9/7-9:33} and Lemma \ref{9/2-13:45} that 
we can apply Lemma \ref{20/10/21/15:57} 
to the second term on the right-hand side of \eqref{eq5-29} as $k=2$, $j=5$, $g=\partial_{a}\varphi(\omega, a) 
\partial_{a}^{2}\varphi(\omega, a)$. Thus, we find that $\partial^{3}\varphi(\omega, a)$ is continuous with respect to $(\omega, a)$ in the $H^{2}(\R \times \T)$-topology. Similarly, we can prove the continuity of the other third order derivatives. 
\end{proof}


\subsection{Derivatives of $\boldsymbol{\mathcal{F}_{\parallel}(\omega, a)}$}
\label{derivative-F-parallel}

In this section, we compute the derivatives of $\mathcal{F}_{\parallel}(\omega, a)$ 
up to the third order.

Observe that 
\begin{equation}\label{eq5-34}
\langle \varphi(\omega, a), \psi_{\omega_{p}}\cos{y} \rangle=a
.
\end{equation}
Then, the following result follows from Lemma \ref{p-d}, 
 \eqref{der-a-h}, \eqref{der-w-h} and \eqref{eq5-34}: 
\begin{proposition}\label{cor1-F-par}
Assume $p>1$. Then, $\mathcal{F}_{\parallel}$ is $C^{1}$ on $(\omega_{p} - \delta_{0}, \omega_{p}+ \delta_{0}) \times (-a_{0}, a_{0})$, and 
the following hold for all $(\omega, a) \in (\omega_{p} - \delta_{0}, \omega_{p}+ \delta_{0}) \times (-a_{0}, a_{0})$: 
\begin{align}
\label{eq5-35}
\partial_{a}\mathcal{F}_{\parallel}(\omega, a)
&=
\langle \mathbf{L}_{+}(\omega, a)\partial_{a}\varphi(\omega, a)
, 
\psi_{\omega_{p}} \cos{y}
\rangle
, 
\\[6pt]
\label{eq5-36}
\partial_{\omega}\mathcal{F}_{\parallel}(\omega, a)
&=
\langle \mathbf{L}_{+}(\omega, a)\partial_{\omega}\varphi(\omega, a), 
\psi_{\omega_{p}} \cos{y}
\rangle 
+ a. 
\end{align}
\end{proposition}


The following result follows from Proposition \ref{cor1-F-par}, Lemma \ref{p-d}, Lemma \ref{d-L} and Proposition \ref{main-lem-1}: 
\begin{proposition}\label{2d-f-para}
Assume $p>1$. Let $\varepsilon_{2}$ be the constant given in Proposition \ref{main-lem-1}  and $0< \varepsilon <\varepsilon_{2}$.  
 Then, $\mathcal{F}_{\parallel}$ is $C^{2}$ on $(\omega_{p} - \delta_{0}, \omega_{p}+ \delta_{0}) \times (-a_{\varepsilon}, a_{\varepsilon})$, and 
the following hold for all $(\omega, a) 
\in (\omega_{p} - \delta_{0},~\omega_{p}+ \delta_{0}) \times 
(-a_{\varepsilon}, a_{\varepsilon})$:
\begin{align} 
\label{eq5-37}
&
\partial_{a}^{2} \mathcal{F}_{\parallel}(\omega, a)
= 
\langle \mathbf{L}_{+}(\omega, a) \partial_{a}^{2}\varphi(\omega, a)
-
V_{2}(\omega, a) \{\partial_{a}\varphi(\omega, a)\}^{2}
, 
\psi_{\omega_{p}} \cos{y}
\rangle 
, 
\\[6pt]
\label{eq5-38}
&\partial_{a} \partial_{\omega} 
\mathcal{F}_{\parallel}(\omega, a)
=
\partial_{\omega} \partial_{a} 
\mathcal{F}_{\parallel}(\omega, a)
\\[6pt]
\nonumber 
&=
1
+
\langle \mathbf{L}_{+}(\omega, a) \partial_{\omega}\partial_{a}\varphi(\omega, a)
-
V_{2}(\omega, a) 
\partial_{a}\varphi(\omega, a)
\partial_{\omega}\varphi(\omega, a)
, 
\psi_{\omega_{p}} \cos{y}
\rangle 
, 
\\[6pt]
\label{eq5-39}
&
\partial_{\omega}^{2} \mathcal{F}_{\parallel}(\omega, a)
=
\langle \mathbf{L}_{+}(\omega, a) \partial_{\omega}^{2}\varphi(\omega, a)
-
V_{2}(\omega, a) \{\partial_{\omega}\varphi(\omega, a)\}^{2}
, 
\psi_{\omega_{p}} \cos{y}
\rangle 
.
\end{align}
\end{proposition}


The following result follows from 
Propositions \ref{main-lem-3} and 
\ref{2d-f-para}, Lemmas \ref{p-d}, 
\ref{d-L} and \ref{lem-phi-v2}: 
\begin{proposition}\label{3d-f-para}
Assume $p>1$. Let $\varepsilon_{3}$ be the same constant given in Proposition {main-lem-3} and $0<\varepsilon< \varepsilon_{3}$. 
Then, $\mathcal{F}_{\parallel}$ is $C^{3}$ on $(\omega_{p} - \delta_{0}, \omega_{p}+ \delta_{0}) \times (-a_{\varepsilon}, a_{\varepsilon})$, and 
the following hold for all $(\omega, a) \in (\omega_{p} - \delta_{0}, \omega_{p}+ \delta_{0}) \times (-a_{\varepsilon}, a_{\varepsilon})$: 
\begin{align} 
\label{dF-aaa}
&\partial_{a}^{3} \mathcal{F}_{\parallel}(\omega, a)
\\[6pt]
\nonumber 
&= 
\langle \mathbf{L}_{+}(\omega, a) \partial_{a}^{3}\varphi(\omega, a)
-
3 V_{2}(\omega, a) \partial_{a}\varphi(\omega, a) \partial_{a}^{2}\varphi(\omega, a)
-
V_{3}(\omega, a) \{\partial_{a}\varphi(\omega, a)\}^{3}
, 
\psi_{\omega_{p}} \cos{y}
\rangle 
, 
\\[6pt]
\label{dF-aaw}
&
\partial_{a}^{2} \partial_{\omega} \mathcal{F}_{\parallel}(\omega, a)
=
\partial_{a} \partial_{\omega} \partial_{a} \mathcal{F}_{\parallel}(\omega, a)
=
\partial_{\omega}\partial_{a}^{2} \mathcal{F}_{\parallel}(\omega, a)
\\[6pt]
\nonumber 
&= 
\langle \mathbf{L}_{+}(\omega, a) \partial_{\omega}\partial_{a}^{2}\varphi(\omega, a)
-
V_{2}(\omega, a) \partial_{\omega}\varphi(\omega, a) \partial_{a}^{2}\varphi(\omega, a)
\\[6pt]
\nonumber 
&\qquad-
2V_{2}(\omega, a) \partial_{a}\varphi(\omega, a) \partial_{\omega}\partial_{a}\varphi(\omega, a)
-
V_{3}(\omega, a) 
\{\partial_{a}\varphi(\omega, a)\}^{2}
\partial_{\omega}\varphi(\omega, a)
, 
\psi_{\omega_{p}} \cos{y}
\rangle 
\\[6pt]
\label{dF-aww}
&
\partial_{a} \partial_{\omega}^{2} \mathcal{F}_{\parallel}(\omega, a)
=
\partial_{\omega} \partial_{a} \partial_{\omega} \mathcal{F}_{\parallel}(\omega, a)
=
\partial_{\omega}^{2} \partial_{a} \mathcal{F}_{\parallel}(\omega, a)
\\[6pt]
\nonumber 
&= 
\langle \mathbf{L}_{+}(\omega, a) \partial_{\omega}^{2}\partial_{a}\varphi(\omega, a)
-
V_{2}(\omega, a) \partial_{a}\varphi(\omega, a) \partial_{\omega}^{2}\varphi(\omega, a)
\\[6pt]
\nonumber 
&\qquad-
2V_{2}(\omega, a) \partial_{\omega}\varphi(\omega, a) \partial_{\omega}\partial_{a}\varphi(\omega, a)
-
V_{3}(\omega, a) 
\{\partial_{\omega}\varphi(\omega, a)\}^{2}
\partial_{a}\varphi(\omega, a)
, 
\psi_{\omega_{p}} \cos{y}
\rangle 
\\[6pt]
\label{dF-www}
&\partial_{\omega}^{3} \mathcal{F}_{\parallel}(\omega, a)
\\[6pt]
\nonumber 
&= 
\langle \mathbf{L}_{+}(\omega, a) \partial_{\omega}^{3}\varphi(\omega, a)
-
3 V_{2}(\omega, a) \partial_{\omega}\varphi(\omega, a) \partial_{\omega}^{2}\varphi(\omega, a)
-
V_{3}(\omega, a) \{\partial_{\omega}\varphi(\omega, a)\}^{3}
, 
\psi_{\omega_{p}} \cos{y}
\rangle 
.
\end{align}
\end{proposition}

\section{Proof of Theorem \ref{main-prop}}\label{21/10/27/11:47}
This section is devoted to the proof 
of Theorem \ref{main-prop}. 
 Throughout this section, for a given $p>1$, let $a_{0}$, $\delta_{0}$ and $\eta \colon (\omega_{p}-\delta_{0}, \omega_{p}+\delta_{0}) \times (-a_{0},a_{0}) \to X_{2}$ be the same as in Lemma \ref{main-lem-2},  $\varphi(\omega,a)$ be the function defined by \eqref{eq2-10}, and $\varepsilon_{3}$ be the constant given in Proposition \ref{main-lem-3}. 
 Furthermore, for $0<\varepsilon <\varepsilon_{3}$, we use $a_{\varepsilon}$ to denote the same constant given in Proposition \ref{lem-up-low-b}.  

We will empoly the argument developed by Crandall and Rabinowitz~\cite{Crandall-Rabinowitz} (see also the proof of Theorem 4 {\rm (ii)} in \cite{KKP}). Let $0<\varepsilon<\varepsilon_{3}$, and define 
\begin{equation}\label{21/11/2/16:4}
D_{\varepsilon}
:= (\omega_{p}-\delta_{0},~\omega_{p}+\delta_{0}) 
\times (-a_{\varepsilon},a_{\varepsilon})
.
\end{equation} 
Furthermore, we introduce the function $g \colon  D_{\varepsilon}\to \mathbb{R}$ as  
\begin{equation} \label{g-eq}
g(\omega, a) 
:= 
\begin{cases}
\dfrac{\mathcal{F}_{\parallel}(\omega, a)}{a}
&\mbox{if $a \neq 0$}, 
\\[6pt]
\p_{a} \mathcal{F}_{\parallel}(\omega, 0) 
& \mbox{if $a = 0$}
.
\end{cases}
\end{equation}


\begin{lemma}\label{21/11/9/11:7}
Assume $p>1$ and let $0<\varepsilon <\varepsilon_{3}$. Then, $g$ is twice differentiable on $D_{\varepsilon}$. Furthermore, the following hold: 
\begin{align}
\label{21/11/7/12:24}
\p_{\omega} g (\omega, 0)
&=
\p_{\omega}\p_{a} \mathcal{F}_{\parallel}(\omega, 0)
=
\p_{a} \p_{\omega} \mathcal{F}_{\parallel}(\omega, 0),
\\[6pt]
\label{21/11/7/14:41}
\p_{\omega}^{2} g (\omega, 0)
&=
\p_{\omega}^{2} \p_{a} 
\mathcal{F}_{\parallel}(\omega, 0)
,
\\[6pt]
\label{21/11/9/11:14}
\p_{a} g (\omega, 0)
&=
\frac{1}{2}\p_{a}^{2} \mathcal{F}_{\parallel}(\omega, 0),
\\[6pt]
\label{21/11/9/11:22}
\p_{\omega}\p_{a} g (\omega, 0)
=\p_{a} \p_{\omega} g (\omega, 0)
&=
\frac{1}{2}\p_{\omega}\p_{a}^{2} \mathcal{F}_{\parallel}(\omega, 0),
\\[6pt]
\label{21/11/9/11:23}
\p_{a}^{2}g(\omega,0)
&=
\frac{1}{3} \p_{a}^{3} \mathcal{F}_{\parallel}(\omega,0)
.
\end{align} 
\end{lemma}
\begin{proof}[Proof of Lemma \ref{21/11/9/11:7}]
Since $\mathcal{F}_{\parallel}$ is $C^{3}$ on $D_{\varepsilon}$ (see Proposition \ref{3d-f-para}), it is obvious that $g$ is twice differentiable at $(\omega,a)\in D_{\varepsilon}$ with $a\neq 0$. 
 Furthermore, we can easily verify \eqref{21/11/7/12:24} and \eqref{21/11/7/14:41}. 

It remains to prove the twice differentiablity of $g$ at $(\omega,0)$  
 and \eqref{21/11/9/11:14} through \eqref{21/11/9/11:23}. 
 The Taylor expansion together with $\mathcal{F}_{\parallel}(\omega, 0)=0$ (see \eqref{eq2-12}) shows that 
\begin{equation}\label{21/11/2/16:33}
\mathcal{F}_{\parallel}(\omega, a)
=
\p_{a} \mathcal{F}_{\parallel}(\omega, 0) a 
+
\frac{1}{2} 
\p_{a}^{2} \mathcal{F}_{\parallel}(\omega, 0) a^{2}
+
\frac{1}{3!}
\p_{a}^{3}\mathcal{F}_{\parallel}(\omega,0)a^{3}
+
o(a^{3})
.
\end{equation}
Furthermore, by Proposition \ref{cor1-F-par}, \eqref{21/11/7/15:47} and the integral of $\cos{y}$ over $[-\pi,\pi]$ being zero, 
 we see that 
\begin{equation}\label{21/11/7/13:35}
\partial_{\omega}\mathcal{F}_{\parallel}(\omega, 0)
=
\langle \mathbf{L}_{+}(\omega, 0)\partial_{\omega}\varphi(\omega, 0),~ 
\psi_{\omega_{p}} \cos{y}
\rangle 
=
\langle -R_{\omega},\psi_{\omega_{p}} \cos{y}
\rangle =0
.
\end{equation}  
Hence, the Taylor expansion together with \eqref{21/11/7/13:35} shows that 
\begin{equation}\label{21/11/7/15:59}
\p_{\omega}\mathcal{F}_{\parallel}(\omega, a)
=
\p_{a} \p_{\omega}\mathcal{F}_{\parallel}(\omega, 0) a 
+
\frac{1}{2} 
\p_{a}^{2} \p_{\omega}\mathcal{F}_{\parallel}(\omega, 0) a^{2}
+
o(a^{2})
.
\end{equation}
Moreover, we will use the Taylor exapnasion of  $\p_{a} \mathcal{F}_{\parallel}(\omega, a)$: 
\begin{equation}\label{21/11/8/12:40}
\p_{a} \mathcal{F}_{\parallel}(\omega, a)
=
\p_{a} \mathcal{F}_{\parallel}(\omega,0)
+
\p_{a}^{2} \mathcal{F}_{\parallel}(\omega,0)a 
+
\frac{1}{2} \p_{a}^{3} \mathcal{F}_{\parallel}(\omega,0)a^{2}
+
o(a^{2})
. 
\end{equation}


The claim \eqref{21/11/9/11:14} follows from \eqref{21/11/2/16:33}: 
\begin{equation}\label{21/11/2/16:42}
\begin{split}
\p_{a} g (\omega, 0)
&=\lim_{a\to 0} \frac{g(\omega,a)-g(\omega,0)}{a} 
=
\lim_{a\to 0} 
\dfrac{\mathcal{F}_{\parallel}(\omega, a)-  a \p_{a} \mathcal{F}_{\parallel}(\omega, 0)  }{a^{2}}
=
\frac{1}{2}\p_{a}^{2} \mathcal{F}_{\parallel}(\omega, 0) 
.
\end{split}
\end{equation} 
Furthermore, \eqref{21/11/2/16:42} shows that
\begin{equation}\label{21/11/7/14:57}
\p_{\omega}\p_{a} g (\omega, 0)
=
\lim_{\delta \to 0} \frac{\p_{a}g(\omega+\delta,0)-\p_{a}g(\omega,0)}{\delta} 
=
\frac{1}{2}\p_{\omega}\p_{a}^{2} \mathcal{F}_{\parallel}(\omega, 0)
.
\end{equation} 
 By \eqref{21/11/7/12:24} and \eqref{21/11/7/15:59}, we see that 
\begin{equation}\label{21/11/7/15:20}
\p_{a} \p_{\omega} g (\omega, 0)
=
\lim_{a\to 0}
\dfrac{ \p_{\omega}\mathcal{F}_{\parallel}(\omega, a)
- 
a \p_{a}\p_{\omega} \mathcal{F}_{\parallel}(\omega, 0)}{a^{2}}
=
\frac{1}{2} 
\p_{a}^{2} \p_{\omega}\mathcal{F}_{\parallel}(\omega, 0)
.
\end{equation}
Then,  \eqref{21/11/9/11:22} follows from \eqref{21/11/7/14:57},  \eqref{21/11/7/15:20} and \eqref{dF-aaw} in Proposition \ref{3d-f-para}.  

It remains to prove \eqref{21/11/9/11:23}.
  Let $(\omega,a) \in D_{\varepsilon}$ with $a\neq 0$. 
Then, by the differentation of quotient and \eqref{21/11/2/16:33}, we see that \begin{equation}\label{21/11/2/16:55}
\begin{split} 
\p_{a} g (\omega, a)
&=
\dfrac{a \p_{a} \mathcal{F}_{\parallel}(\omega, a)- \mathcal{F}_{\parallel}(\omega, a)}{a^{2}}
\\[6pt]
&=
\dfrac{\p_{a} \mathcal{F}_{\parallel}(\omega, a)- 
\p_{a}\mathcal{F}_{\parallel}(\omega, 0) }{a}
-
\frac{1}{2} \p_{a}^{2}\mathcal{F}_{\parallel}(\omega, 0)
-
\frac{1}{3!}\p_{a}^{3} \mathcal{F}_{\parallel}(\omega,0)a 
+ 
o(a)
\\[6pt]
&=
\int_{0}^{1} 
\p_{a}^{2} \mathcal{F}_{\parallel}(\omega, \theta a) \,d\theta 
-
\frac{1}{2}\p_{a}^{2} \mathcal{F}_{\parallel}(\omega, 0) 
-
\frac{1}{3!}\p_{a}^{3} \mathcal{F}_{\parallel}(\omega,0)a 
+ 
o(a).
\end{split} 
\end{equation}
Furthermore, by \eqref{21/11/2/16:55} and \eqref{21/11/2/16:42}, 
 we see that 
\begin{equation}\label{21/11/7/14:58}
\begin{split} 
\p_{a}^{2} g (\omega, 0)
&=
\lim_{a\to 0} \frac{\p_{a}g(\omega,a) - \p_{a}g(\omega,0)}{a}
\\[6pt]
&=
\lim_{a\to 0} \frac{1}{a}
\Big\{ 
\int_{0}^{1} 
\p_{a}^{2} \mathcal{F}_{\parallel}(\omega, \theta a) \,d\theta 
-
\p_{a}^{2} \mathcal{F}_{\parallel}(\omega, 0) 
-
\frac{1}{3!}\p_{a}^{3} \mathcal{F}_{\parallel}(\omega,0)a 
+ 
o(a)
\Big\}
\\[6pt]
&=
\lim_{a\to 0}
\int_{0}^{1} \int_{0}^{1} 
\p_{a}^{3} \mathcal{F}_{\parallel}(\omega, \theta \kappa a) \,d\kappa d\theta 
-
\frac{1}{3!}\p_{a}^{3} \mathcal{F}_{\parallel}(\omega,0) 
\\[6pt]
&=
\int_{0}^{1}\int_{0}^{1} \p_{a}^{3} \mathcal{F}_{\parallel}(\omega, 0) \,d\kappa d\theta 
-
\frac{1}{6}\p_{a}^{3} \mathcal{F}_{\parallel}(\omega,0) 
=
\frac{1}{3} \p_{a}^{3} \mathcal{F}_{\parallel}(\omega,0)
.
\end{split} 
\end{equation}
Thus, we have completed the proof. 
\end{proof} 


\begin{lemma}\label{g-pro}
Assume $p>1$, and let $0<\varepsilon <\varepsilon_{3}$. 
 Then, the function $g$ is $C^{2}$ on $D_{\varepsilon}$. 
 Furthermore, the following hold:   
\begin{equation}\label{21/7/31/17:26}
g(\omega_{p}, 0) = 0, 
\qquad 
\p_{a} g(\omega_{p}, 0) = 0 ,
\qquad 
\p_{\omega} g(\omega_{p}, 0) 
= 
-\frac{1}{\omega_{p}} <0
.
\end{equation} 
\end{lemma}
\begin{proof}[Proof of Lemma \ref{g-pro}]
We shall show that $g$ is $C^{2}$ on $D_{\varepsilon}$. 
 Since $\mathcal{F}_{\parallel}$ is $C^{3}$ on $D_{\varepsilon}$ (see Proposition \ref{3d-f-para}), it is obvious that $g$ is twice continuously differentiable at $(\omega,a)\in D_{\varepsilon}$ with $a\neq 0$. 
 Furthermore, by Lemma \ref{21/11/9/11:7}, it suffices to prove the continuity of the second derivatives $\p_{\omega}^{2}g$, $\p_{\omega} \p_{a}g$ and $\p_{a}^{2}g$ at $(\omega, 0)$. 

By \eqref{phi-a=0}, \eqref{21/8/1/15:10} and $\varphi$ being $C^{2}$ on $D_{\varepsilon}$ (see Proposition \ref{main-lem-1}), we see that 
both $\mathbf{L}_{+}(\omega,0) 
\partial_{\omega}^{2}\varphi(\omega, 0)$ and 
 $V_{2}(\omega, 0) \{\partial_{\omega}\varphi(\omega, 0)\}^{2}$
 are independent of $y$. 
 Hence, \eqref{eq5-39} in Proposition \ref{2d-f-para} together with 
the integral of $\cos{y}$ over $[-\pi,\pi]$ being zero 
 shows that  
\begin{equation}\label{21/11/10/11:18}
\partial_{\omega}^{2} \mathcal{F}_{\parallel}(\omega, 0)
=
\langle \mathbf{L}_{+}(\omega,0) 
\partial_{\omega}^{2}\varphi(\omega, 0)
-
V_{2}(\omega, 0) \{\partial_{\omega}\varphi(\omega, 0)\}^{2}
, 
\psi_{\omega_{p}} \cos{y}
\rangle 
=0
.
\end{equation} 
Then, by the definition of $g$, \eqref{21/11/10/11:18} and \eqref{21/11/7/14:41} in Lemma \ref{21/11/9/11:7}, we see that 
\begin{equation}\label{21/11/9/17:53}
\lim_{a\to 0} \p_{\omega}^{2}g(\omega, a)
=
\lim_{a\to 0} \frac{\p_{\omega}^{2} \mathcal{F}_{\parallel}(\omega,a)}{a}
=
\p_{a}\p_{\omega}^{2}  
\mathcal{F}_{\parallel}(\omega, 0)
=
\p_{\omega}^{2} g (\omega, 0)
.
\end{equation}


Next, we consider $\p_{\omega}\p_{a} g$. 
 By the differentation of quotient,  the Taylor expansion of $\p_{\omega}\mathcal{F}_{\parallel}$ (see \eqref{21/11/7/15:59}), and \eqref{21/11/9/11:22} in Lemma \ref{21/11/9/11:7}, 
 we see that 
\begin{equation}\label{21/11/10/11:29}
\begin{split} 
&\lim_{a\to 0}
\p_{\omega}\p_{a} g (\omega, a)
=
\lim_{a\to 0}
\p_{\omega}
\Big\{
\dfrac{a \p_{a} \mathcal{F}_{\parallel}(\omega, a)- \mathcal{F}_{\parallel}(\omega, a)}{a^{2}}
\Big\}
\\[6pt]
&=
\lim_{a\to 0}
\Big\{
\dfrac{a \p_{a} \p_{\omega}\mathcal{F}_{\parallel}(\omega, a)
- 
a \p_{a} \p_{\omega}\mathcal{F}_{\parallel}(\omega, 0)}{a^{2}}
-\frac{1}{2}\p_{a}^{2}\p_{\omega}\mathcal{F}_{\parallel}(\omega,0) 
+
o_{a}(1)
\Big\}
\\[6pt]
&=
\frac{1}{2}\p_{a}^{2}\p_{\omega}\mathcal{F}_{\parallel}(\omega,0) 
=
\p_{\omega}\p_{a}g(\omega,0).
\end{split} 
\end{equation}


Finally, we consider $\p_{a}^{2}g$. 
 By the definition of $g$, \eqref{21/11/2/16:33} and \eqref{21/11/8/12:40}, 
 we can verify that for any $(\omega,a)\in D_{\varepsilon}$ with $a\neq 0$, 
\begin{equation}\label{21/11/7/16:35}
\begin{split} 
\p_{a}^{2} g (\omega, a)
&=
\dfrac{a^{2} \p_{a}^{2} \mathcal{F}_{\parallel}(\omega, a)- 2 a \p_{a} \mathcal{F}_{\parallel}(\omega, a) + 2 \mathcal{F}_{\parallel}(\omega,a) 
}{a^{3}}
\\[6pt]
&=
\dfrac{a^{2} \p_{a}^{2} \mathcal{F}_{\parallel}(\omega, a)- 2 a \p_{a} \mathcal{F}_{\parallel}(\omega, a) + 2 a \p_{a} \mathcal{F}_{\parallel}(\omega,0)
+ a^{2} \p_{a}^{2}\mathcal{F}_{\parallel}(\omega,0)}{a^{3}} 
\\[6pt]
&\quad +
\frac{1}{3}\p_{a}^{3}\mathcal{F}_{\parallel}(\omega,0)+o_{a}(1)
\\[6pt]
&=
a^{-1} 
\p_{a}^{2} \mathcal{F}_{\parallel}(\omega, a)
-
a^{-1}
\p_{a}^{2} \mathcal{F}_{\parallel}(\omega,0) 
-
\frac{2}{3}\p_{a}^{3}\mathcal{F}_{\parallel}(\omega,0)+o_{a}(1)
\\[6pt]
&=
\int_{0}^{1} \p_{a}^{3} \mathcal{F}_{\parallel}(\omega, \theta a)\,d\theta 
-
\frac{2}{3}\p_{a}^{3}\mathcal{F}_{\parallel}(\omega,0)+o_{a}(1)
.
\end{split} 
\end{equation}
Then, the continuity of $\p_{a}^{2}g$ at $(\omega,0)$ follows from \eqref{21/11/7/16:35} and \eqref{21/11/9/11:23} in Lemma \ref{21/11/9/11:7}:  
\begin{equation}\label{21/11/9/11:51}
\lim_{a\to 0}\p_{a}^{2} g (\omega, a)
=
\int_{0}^{1} \p_{a}^{3} \mathcal{F}_{\parallel}(\omega, 0)\,d\theta 
-
\frac{2}{3}\p_{a}^{3}\mathcal{F}_{\parallel}(\omega,0)
=
\p_{a}^{2}g(\omega,0)
. 
\end{equation}
Thus, we have proved that $g$ is $C^{2}$ on $D_{\varepsilon}$. 


We shall prove \eqref{21/7/31/17:26}.  
  Note that the following follows from the self-adjointness of $\mathbf{L}_{+}(\omega, a)$ on $L^{2}(\R\times \T)$, \eqref{21/8/1/15:10} and \eqref{eq1-12}:   
\begin{equation}\label{21/11/7/9:45}
\langle 
\mathbf{L}_{+}(\omega_{p}, 0) u,~
\psi_{\omega_{p}} \cos{y} 
\rangle=0
\quad 
\mbox{for all $u\in H^{2}(\R\times \T)$}
.  
\end{equation}
Then, by \eqref{eq5-35} in Proposition \ref{cor1-F-par}, and \eqref{21/11/7/9:45}, 
 we see that 
\begin{equation}\label{21/11/10/13:23} 
g(\omega_{p},0)
=
\langle 
\mathbf{L}_{+}(\omega_{p},0) \p_{a}\varphi(\omega_{p},0),~\psi_{\omega_{p}}\cos{y} \rangle 
=0
.
\end{equation} 
By \eqref{21/11/9/11:14},  \eqref{eq5-37} in Proposition \ref{2d-f-para},  $V_{2}(\omega_{p},0)=p(p-1) R_{\omega_{p}}^{p-2}$ (see \eqref{def-V} and \eqref{phi-a=0}), $\partial_{a}\varphi(\omega_{p}, 0)=\psi_{\omega_{p}} \cos{y}$ (see \eqref{21/8/1/15:1}), \eqref{21/11/7/9:45} and the integral of $(\cos{y})^{3}$ over $[-\pi,\pi]$ being zero , we see that 
\begin{equation}\label{21/11/6/17:33}
\begin{split}
\p_{a} g (\omega_{p}, 0) 
&=
\frac{1}{2}
\langle 
\mathbf{L}_{+}(\omega_{p}, 0) 
\p_{a}^{2}\varphi(\omega_{p},0)
-
p(p-1) R_{\omega_{p}}^{p-2}\{ \psi_{\omega_{p}}\cos{y} \}^{2}
,~ 
\psi_{\omega_{p}} \cos{y}
\rangle 
\\[6pt]
&=
-\frac{p(p-1)}{2}
\int_{\R} R_{\omega_{p}}^{p-2}\psi_{\omega_{p}}^{3} \,dx 
\int_{\T}(\cos{y})^{3}\,dy 
=0
.
\end{split} 
\end{equation}
By \eqref{21/11/7/12:24}, \eqref{eq5-38} in Proposition \ref{2d-f-para}, 
 \eqref{21/11/7/9:45}, 
 \eqref{phi-a=0}, \eqref{21/8/1/15:1}, the definition of $\psi_{\omega_{p}}$ (see \eqref{eq1-10}), 
 \eqref{21/11/6/10:6} and 
 \eqref{21/11/6/14:34}, we see that 
\begin{equation}\label{21/11/1/10:35}
\begin{split} 
&
\partial_{\omega} g (\omega_{p}, 0) = 
\partial_{\omega} 
\partial_{a} \mathcal{F}_{\parallel} 
(\omega_{p}, 0)
\\[6pt]
&=
1 +
\langle \mathbf{L}_{+}(\omega_{p}, 0) \partial_{\omega}\partial_{a}\varphi(\omega_{p}, 0)- V_{2}(\omega_{p}, 0) \p_{a}\varphi(\omega_{p}, 0)
 \p_{\omega} \varphi (\omega_{p},0),~ 
\psi_{\omega_{p}} \cos{y}
\rangle 
\\[6pt]
&=
1 - 
p(p-1) 
\langle R_{\omega_{p}}^{p-2} 
(\psi_{\omega_{p}} \cos{y})
\p_{\omega}R_{\omega}|_{\omega=\omega_{p}},~ 
\psi_{\omega_{p}} \cos{y}
\rangle 
\\[6pt]
&=
1 - 
\frac{p(p-1)}{\|R_{\omega_{p}} \|_{L^{p+1}(\R)}^{p+1}} 
\int_{\R}
R_{\omega_{p}}^{2p-1} 
\p_{\omega}R_{\omega}|_{\omega=\omega_{p}} 
\,dx 
\\[6pt]
&=
1 - 
\frac{p(p-1)}{\|R_{\omega_{p}} \|_{L^{p+1}(\R)}^{p+1}} 
\frac{3p+1}{4p(p-1)}
\omega_{p}^{-1} \int_{\R}
R_{\omega_{p}}^{2p} 
\,dx 
\\[6pt]
&=
1 - 
\frac{1}{\|R_{\omega_{p}} \|_{L^{p+1}(\R)}^{p+1}} 
\frac{(p+1)^{2}}{4}
\int_{\mathbb{R}} 
R_{\omega_{p}}^{p+1}
=
-\frac{(p-1)(p+3)}{4}
=
-\frac{1}{\omega_{p}}
.
\end{split}
\end{equation}
Thus, we have completed the proof of the lemma. 
\end{proof}

Now, we are in a position to prove Theorem \ref{main-prop}. 
\begin{proof}[Proof of Theorem \ref{main-prop}]
The first step to prove Theorem \ref{main-prop} is to find a curve $\omega(a)$ parametrized by $a$ such that $(\omega(a),a)$ satisfies the bifurcation equation \eqref{bifur-eq}:
\begin{equation}\label{21/11/10/13:58}
\mathcal{F}_{\parallel}(\omega(a), a) 
= \langle  
\mathcal{F}(\omega, \varphi(\omega(a),a)),~\psi_{\omega_{p}}\cos{y}
\rangle  
= 0
.
\end{equation}
The implicit function theorem together with  Lemma \ref{g-pro} shows that 
 there exist $a_{*}>0$, $\delta_{*}>0$ and a $C^{2}$-curve $\omega(\cdot) \colon a \in (-a_{*},a_{*}) \mapsto \omega(a)\in (\omega_{p}-\delta_{*}, \omega_{p}+\delta_{*})$ with the following properties:
\begin{enumerate}
\item  
\begin{equation} \label{g-eq0}
g(\omega(a), a) = 0,
\qquad 
\omega(0)=\omega_{p},
\qquad 
\frac{d \omega}{da}(a)
=
-\frac{\p_{a}g(\omega(a),a)}{\p_{\omega}g(\omega(a),a)}
.
\end{equation}
\item~Let $Z_{p}$ be the set of zeros of $g=0$ in $(\omega_{p}-\delta_{*}, \omega_{p}+\delta_{*})\times (-a_{*},a_{*})$. Then,   
\begin{equation}\label{21/11/12/16:25}
Z_{p}=
\{(\omega(a),a) \colon a \in (-a_{*},a_{*})\}
.
\end{equation} 
\end{enumerate}

Next, we define the function $Q \colon (-a_{*},a_{*}) \to H^{2}(\R \times \T)$ by 
\begin{equation}\label{21/11/10/14:22}
Q(a)
:=\varphi(\omega(a),a)
=R_{\omega_{p}}+ a \psi_{\omega_{p}} \cos{y} + \eta(\omega(a),a)
.
\end{equation}
Note that $Q$ is $C^{2}$ on $(-a_{*},a_{*})$ in $H^{2}(\R \times \T)$.  
 
Since $g(\omega,a)=0$ with $a\neq 0$ implies that $\mathcal{F}_{\parallel}(\omega,a)=0$, the first claim of Theorem \ref{main-prop} follows from \eqref{21/11/12/16:25} and Lemma \ref{main-lem-2}. 
 Furthermore, the claim \eqref{21/11/10/14:59} follows immediately from \eqref{g-eq0} and \eqref{21/7/31/17:26}.


We shall prove \eqref{main-omega1}, \eqref{expan-phi} and \eqref{eq1-14}.   
 
 First, note that the integral of $(\cos{y})^{3}$ over $[-\pi,\pi]$ being zero shows that 
\begin{equation}\label{21/11/12/14:25}
R_{\omega_{p}}^{p-2}\{ \psi_{\omega_{p}}\cos{y} \}^{2}\in X_{2}.
\end{equation} 
Furthermore, by \eqref{der2-h-a} in Proposition \ref{main-lem-1}, \eqref{21/8/1/15:15}, 
 $V_{2}(\omega_{p},0)=p(p-1) R_{\omega_{p}}^{p-2}$ (see \eqref{def-V} and \eqref{phi-a=0}), \eqref{21/8/1/15:1} and \eqref{21/11/12/14:25}, we see that 
\begin{equation}\label{21/11/11/11:53}
\partial_{a}^{2} \varphi (\omega_{p}, 0)
=
p(p-1) 
\big\{ 
P_{\perp} \partial_{u}\mathcal{F}(\omega_{p},R_{\omega_{p}})|_{X_{2}}
\big\}^{-1}
\big(
R_{\omega_{p}}^{p-2}
\{ \psi_{\omega_{p}}\cos{y} \}^{2}
\big)
.
\end{equation} 


 Differentiating both sides of the last equation in \eqref{g-eq0}, 
and using \eqref{21/11/10/14:59}, \eqref{21/7/31/17:26} and \eqref{21/11/9/11:23} in Lemma \ref{21/11/9/11:7},  
 we see that 
\begin{equation}\label{21/11/12/10:56}
\frac{d^{2} \omega}{da^{2}}(0)
=
-\frac{\omega_{p}}{3} 
\p_{a}^{3} \mathcal{F}_{\parallel}(\omega_{p}, a)
.
\end{equation}
By \eqref{dF-aaa} in Proposition \ref{3d-f-para}, \eqref{21/11/7/9:45},
 $V_{2}(\omega_{p},0)=p(p-1) R_{\omega_{p}}^{p-2}$ and 
 $V_{3}(\omega_{p},0)=p(p-1)(p-2) R_{\omega_{p}}^{p-3}$ (see \eqref{def-V} and \eqref{phi-a=0}), 
 $\partial_{a}\varphi(\omega_{p}, 0)=\psi_{\omega_{p}} \cos{y}$ (see \eqref{21/8/1/15:1}), \eqref{21/11/11/11:53} and \eqref{21/11/12/14:25}, 
 we see that 
\begin{equation}\label{21/11/12/14:6}
\begin{split}
&\partial_{a}^{3} \mathcal{F}_{\parallel}(\omega_{p}, 0)
\\[6pt]
&= 
\langle 
-
3 V_{2}(\omega_{p}, 0) 
\partial_{a}\varphi(\omega_{p}, 0) 
\partial_{a}^{2}\varphi(\omega_{p}, 0)
-
V_{3}(\omega_{p}, 0) \{\partial_{a}\varphi(\omega_{p}, 0)\}^{3}
,~ 
\psi_{\omega_{p}} \cos{y}
\rangle 
\\[6pt]
&= 
-
3 \{ p(p-1)\}^{2}
\langle 
R_{\omega_{p}}^{p-2} 
\psi_{\omega_{p}}\cos{y} \, 
\mathbf{T}(\omega_{p}, 0)^{-1}
\big(
R_{\omega_{p}}^{p-2}
\{ \psi_{\omega_{p}}\cos{y} \}^{2}
\big)
,~ 
\psi_{\omega_{p}} \cos{y}
\rangle 
\\[6pt]
&\quad 
-p(p-1)(p-2)
\langle 
R_{\omega_{p}}^{p-3} 
\{\psi_{\omega_{p}}\cos{y}\}^{3}
,~ 
\psi_{\omega_{p}} \cos{y}
\rangle 
\\[6pt]
&= 
-
3 \{ p(p-1)\}^{2}
\langle 
\mathbf{T}(\omega_{p}, 0)^{-1}
\big(
R_{\omega_{p}}^{p-2}
\{ \psi_{\omega_{p}}\cos{y} \}^{2}
\big)
,~ 
R_{\omega_{p}}^{p-2} 
\{ \psi_{\omega_{p}}\cos{y}\}^{2} 
\rangle 
\\[6pt]
&\quad 
-p(p-1)(p-2)
\langle 
R_{\omega_{p}}^{p-3} 
\{\psi_{\omega_{p}}\cos{y}\}^{3}
,~ 
\psi_{\omega_{p}} \cos{y}
\rangle 
.
\end{split}
\end{equation} 
Plugging \eqref{21/11/12/14:6} into \eqref{21/11/12/10:56}, we obtain \eqref{main-omega1}.


Observe that  
\begin{equation}\label{21/11/11/11:7}
\frac{dQ}{da}(a)
=
\p_{\omega}\varphi(\omega(a),a)\frac{d\omega}{da}(a) 
+
\p_{a}\varphi(\omega(a),a) 
.
\end{equation}
Then, by \eqref{21/11/11/11:7}, \eqref{21/11/10/14:59} 
 and \eqref{21/8/1/15:1},  we see that 
\begin{equation}\label{21/11/11/11:22}
\frac{dQ}{da}(0)
=
\p_{a}\varphi(\omega_{p},0) 
=
\psi_{\omega_{p}}\cos{y}
.
\end{equation}
Furthermore, differentiating both sides of \eqref{21/11/11/11:7}, 
 and using \eqref{21/11/10/14:59} and 
 \eqref{phi-a=0}, 
 we see that 
\begin{equation}\label{21/11/11/11:12}
\begin{split}
\frac{d^{2}Q}{da^{2}}(0)
&=
\p_{\omega}^{2}\varphi(\omega(0),0) \Big( \frac{d\omega}{da}(0) \Big)^{2}
+
\p_{\omega}\varphi(\omega(0),0)\frac{d^{2}\omega}{da^{2}}(0) 
\\[6pt]
&\quad +
\p_{\omega}\p_{a}\varphi(\omega(0),0)\frac{d\omega}{da}(0)
+
\p_{a}^{2}\varphi(\omega(0),0)
\\[6pt]
&=
\p_{\omega}R_{\omega}|_{\omega=\omega_{p}} \frac{d^{2}\omega}{da^{2}}(0) 
+
\p_{a}^{2}\varphi(\omega_{p},0) 
.
\end{split} 
\end{equation}
Then, the Taylor expansion together with $Q(0)=\varphi(\omega(0),0)=R_{\omega_{p}}$ (see \eqref{phi-a=0}), \eqref{21/11/11/11:22}
 and \eqref{21/11/11/11:12}  shows that the following holds in $H^{2}(\R\times \T)$: 
\begin{equation}\label{21/11/10/14:41}
\begin{split}
Q(a)
=
R_{\omega_{p}} 
+
a \psi_{\omega_{p}}\cos{y} 
+
\frac{1}{2}a^{2} 
\big\{ 
\p_{\omega}R_{\omega}|_{\omega=\omega_{p}} \frac{d^{2}\omega}{da^{2}}(0)  
+
\p_{a}^{2}\varphi(\omega_{p},0)
\big\} 
+
o(a^{2})
.
\end{split} 
\end{equation}
Clearly, this shows that \eqref{expan-phi} holds.  


It remains to prove the last claim \eqref{eq1-14}. 
 By \eqref{21/11/10/14:41}, $\|\psi_{\omega_{p}}\cos{y}\|_{L^{2}(\R\times \T)}=1$ and the integral of $\cos{y}$ over $[-\pi,\pi]$ being zero, we see that 
\begin{equation}\label{21/11/10/16:29}
\begin{split}
\|Q(a)\|_{L^{2}(\R\times \T)}^{2}
&=
\|R_{\omega_{p}}\|_{L^{2}(\R\times \T)}^{2}
+
a^{2} 
\\[6pt]
&\quad +
a^{2} \frac{d^{2}\omega}{da^{2}}(0) 
\langle 
R_{\omega_{p}},~
\p_{\omega}R_{\omega}|_{\omega=\omega_{p}}  
\rangle 
+
a^{2}
\langle 
R_{\omega_{p}},~
\p_{a}^{2}\varphi(\omega_{p},0)
\rangle 
+
o(a^{2})
\end{split}
\end{equation}
Recall that $R_{\omega}, \p_{\omega}R_{\omega} \in X_{2}$ (see \eqref{21/7/21/16:37}) and 
 $\p_{u}\mathcal{F}(\omega_{p},R_{\omega_{p}}) \colon X_{2}\to Y_{2}$ is 
 bijective (see \eqref{21/7/24/16:40} and \eqref{orthgo-1}). 
Then, by \eqref{21/11/11/11:53}, \eqref{21/11/7/15:47} 
 and the same computation as \eqref{21/11/1/10:35}, we see that  
\begin{equation}\label{21/11/11/14:59}
\begin{split} 
&\langle
R_{\omega_{p}},~ 
\partial_{a}^{2} \varphi (\omega_{p}, 0)
\rangle
\\[6pt]
&=
p(p-1)
\langle 
\big\{ 
P_{\perp} \partial_{u}\mathcal{F}(\omega_{p},R_{\omega_{p}})|_{X_{2}}
\big\}^{-1}
R_{\omega_{p}},~ 
R_{\omega_{p}}^{p-2}
\{ \psi_{\omega_{p}}\cos{y} \}^{2}
\rangle 
\\[6pt]
&=
-p(p-1)
\langle 
\p_{\omega} R_{\omega}|_{\omega=\omega_{p}},~ 
R_{\omega_{p}}^{p-2}
\{ \psi_{\omega_{p}}\cos{y} \}^{2}
\rangle 
=
-1-\frac{1}{\omega_{p}}
.
\end{split}
\end{equation} 
Plugging \eqref{21/11/11/14:59} into \eqref{21/11/10/16:29}, and using \eqref{21/11/6/10:6}, we obatin \eqref{eq1-14}.
 Thus, we have completed the proof of the theorem. 
\end{proof}


\begin{thank}
T.A was supported 
by JSPS KAKENHI Grant Number 20K03697.
Y.B was supported by PIMS grant 
and NSERC grant (371637-2019).
S.I was supported by NSERC grant 
(371637-2019). 
H.K. was supported by JSPS KAKENHI 
Grant Number JP20K03706.

\end{thank}

\bibliographystyle{plain}

\begin{thebibliography}{99}

\bibitem{BN}
H. Berestycki and L. Nirenberg, 
Some qualitative properties of solutions of semilinear 
elliptic equations in cylindrical domains. 
Analysis, et cetera. Academic Press, 1990. 115--164.

\bibitem{Chang-Gustaf-Nakanishi-Tsai}
S.-M. Chang, S. Gustafson, K. Nakanishi and 
T.-P. Tsai,
Spectra of linearized operators for NLS solitary waves. 
SIAM J. Math. Anal. {\bf 39} (2007/08), 1070--1111. 



\bibitem{Crandall-Rabinowitz}
M. G. Crandall and 
P. H. Rabinowitz, 
Bifurcation from simple eigenvalues.
J. Funct. Anal. 
\textbf{8} (1971) 321--340. 












\bibitem{KKP}
E. Kirr, P. G. Kevrekidis and D. E. Pelinovsky, 
Symmetry-breaking bifurcation in the nonlinear 
Schr\"{o}dinger 
equation with symmetric potentials,
Comm. Math. Phys. {\bf 308} (2011), 795--844.



\bibitem{Lieb-Loss}
 E.H. Lieb and M. Loss,  
 ANALYSIS, second edition, 
 American Mathematical Society (2001).

















\bibitem{Yamazaki2}
Y. Yamazaki,
Stability for line standing waves 
near the bifurcation point for nonlinear 
Schr\"{o}dinger equations. 
Kodai Math. J. {\bf 38} (2015), 65--96. 







\end{thebibliography}

\noindent 
Takafumi Akahori
\\ 
Faculty of Engineering
\\
Shizuoka University
\\
Jyohoku 3-5-1, Hamamatsu-Shi, 
Shizuoka, 
432-8561, Japan
\\ 
E-mail: akahori.takafumi@shizuoka.ac.jp

\vspace{0.5cm}

\noindent
Yakine Bahri
\\
Department of Mathematics and Statistics
\\ 
University of Victoria 
\\
3800 Finnerty Road, Victoria, B.C., Canada V8P 5C2 
\\
E-mail: ybahri@uvic.ca

\vspace{0.5cm}

\noindent
Slim Ibrahim
\\
Department of Mathematics and Statistics
\\ 
University of Victoria 
\\
3800 Finnerty Road, Victoria, B.C., Canada V8P 5C2 
\\
E-mail: ibrahims@uvic.ca

\vspace{0.5cm}

\noindent
Hiroaki Kikuchi
\\
Department of Mathematics 
\\
Tsuda University 
\\
2-1-1 Tsuda-machi, Kodaira-shi, Tokyo 187-8577, JAPAN
\\
E-mail: hiroaki@tsuda.ac.jp

\end{document}